\documentclass{amsart}

\usepackage{amsmath,amssymb}
\usepackage{amsthm}
\usepackage{enumerate}
\usepackage{url}
\usepackage{color}
\usepackage{mathtools}
\usepackage[titletoc]{appendix}

\newtheorem{theorem}{Theorem}[section]
\newtheorem{fact}[theorem]{Fact}
\newtheorem{lemma}[theorem]{Lemma}
\newtheorem{corollary}[theorem]{Corollary}
\newtheorem{proposition}[theorem]{Proposition}

\newtheorem{Claim}[theorem]{Claim}
\newtheorem{alphatheorem}{Theorem}

\theoremstyle{definition}
\newtheorem{definition}[theorem]{Definition}
\newtheorem{example}[theorem]{Example}

\newtheorem{remark}[theorem]{Remark}
\newtheorem*{remark*}{Remark}

\newcommand{\abar}{\bar{a}}
\newcommand{\bbar}{\bar{b}}
\newcommand{\cbar}{\bar{c}}
\newcommand{\dbar}{\bar{d}}

\newcommand{\xbar}{\bar{x}}
\newcommand{\ybar}{\bar{y}}
\newcommand{\zbar}{\bar{z}}
\newcommand{\nv}{\text{-}}

\def\seq{\subseteq}

\def\Def{\operatorname{Def}}

\def\diam{\operatorname{diam}}
\def\str{\operatorname{str}}
\def\psd{\operatorname{psd}}
\def\err{\operatorname{err}}
\def\supp{\operatorname{supp}}

\def\CB{\operatorname{CB}}
\def\CBm{\operatorname{CBm}}
\def\Th{\operatorname{Th}}
\def\N{\mathbb{N}}
\def\Z{\mathbb{Z}}
\def\Q{\mathbb{Q}}
\def\R{\mathbb{R}}
\def\cC{\mathcal{C}}
\def\cL{\mathcal{L}}
\def\cU{\mathcal{U}}

\def\cA{\mathcal{A}}
\def\cB{\mathcal{B}}
\def\cP{\mathcal{P}}

\def\cF{\mathcal{F}}

\def\cL{\mathcal{L}}

\newcommand{\mand}{\makebox[.4in]{and}}

\def\lf{\mathfrak{f}}

\renewcommand{\phi}{\varphi}
\renewcommand{\emptyset}{\varnothing}
\renewcommand{\epsilon}{\varepsilon}

\newcommand{\inv}{^{\text{-}1}}

\def\Ind{\setbox0=\hbox{$x$}\kern\wd0\hbox to 0pt{\hss$\mid$\hss}
\lower.9\ht0\hbox to 0pt{\hss$\smile$\hss}\kern\wd0}

\def\Notind{\setbox0=\hbox{$x$}\kern\wd0\hbox to 0pt{\mathchardef
\nn=12854\hss$\nn$\kern1.4\wd0\hss}\hbox to
0pt{\hss$\mid$\hss}\lower.9\ht0 \hbox to 0pt{\hss$\smile$\hss}\kern\wd0}

\newcommand{\dotminus}{ 
\!\!\buildrel\textstyle~.\over{\hbox{ 
\vrule height3pt depth0pt width0pt}{\smash-} 
}}

\makeatletter
\@namedef{subjclassname@2020}{%
  \textup{2020} Mathematics Subject Classification}
\makeatother

\title{Continuous stable regularity}
\date{December 7, 2023}

\author[N. Chavarria]{Nicolas Chavarria}

\address{Department of Mathematics\\
University of Notre Dame\\
Notre Dame IN 46656\\
 USA}
\email{nchavarr@nd.edu}

\author[G. Conant]{Gabriel Conant}

\address{Department of Mathematics\\
The Ohio State University\\
Columbus, OH, 43210, USA}
\email{conant.38@osu.edu}

\author[A. Pillay]{Anand Pillay}

\address{Department of Mathematics\\
University of Notre Dame\\
Notre Dame IN 46656\\
 USA}
\email{apillay@nd.edu}

\subjclass[2020]{03C45, 03C66, 05C75}

\begin{document}

\begin{abstract}
We prove an analytic version of the stable graph regularity lemma from \cite{MaSh}, which applies to stable functions $f\colon V\times W\to [0,1]$. Our methods involve continuous model theory and, in particular,  results on the structure of local Keisler measures for  stable continuous formulas. Along the way, we develop some basic tools around ultraproducts of metric structures and linear functionals on continuous formulas, and we also describe several concrete families of examples of stable functions.
\end{abstract}

\maketitle

\section*{Introduction}

Szemer\'{e}di's regularity lemma \cite{SzemRL} is a structure theorem for arbitrary finite graphs, and has become a fundamental tool in graph theory (see, e.g., \cite{KomSim}). In the setting of a bipartite graph $(V,W;E)$ (which will be our focus), the regularity lemma says roughly that $V$ and $W$ can be partitioned into a small number (depending only on a fixed  $\epsilon$) of sets $V_{i},W_j$  such that for all $i,j$ outside of a small number of ``irregular pairs", the graph $(V_{i}, W_{j}; E\cap(V_{i}\times W_{j}))$ is $\epsilon$-regular, meaning that sufficiently large induced subgraphs have a common edge density up to error at most $\epsilon$. (See Section \ref{sec:bipartite} for discussion of how to reconcile the setting of bipartite graphs with unpartitioned graphs $(V;E)$.) 

The regularity lemma for graphs can be recast and generalized as a decomposition theorem for  functions $f:V\times W \to [0,1]$, which is sometimes called the \emph{analytic} form of Szemer\'{e}di regularity.  See \cite{Green-Tao, LoSz, TaoSR}, for example.   In this regime, $f$ is decomposed as the sum of a ``structured" part,  a ``pseudorandom" part, and an ``error" part.  In the special case where $f$ is $\{0,1\}$-valued, we are back in the setting of graphs, and it is explained in several places how one recovers the usual statement of Szemer\'{e}di regularity from the analytic version (e.g., \cite[Lemma 2.11]{TaoSR}). See Section \ref{sec:analytic} for further discussion as well as new results in the ``stable" case.

It is rather natural to try to improve the conclusion of Szemer\'{e}di's  regularity lemma by placing  restrictions on the class of finite graphs under consideration. In their seminal paper, Malliaris and Shelah  \cite{MaSh} considered the restriction to ``$k$-stable graphs", namely graphs that omit the $k$-half graph  $([k],[k] ;\leq)$.  The improvements in the conclusions involved  better bounds (as a function of $\epsilon$), no irregular pairs, and ``$\epsilon$-homogeneity" replacing $\epsilon$-regularity.  Here by $\epsilon$-homogeneity of a bipartite graph $(V,W;E)$, we mean that either $|E| \geq (1-\epsilon)|V||W|$ or $|E|\leq \epsilon|V||W|$.

The aim of the current paper is to prove an analytic version of stable graph regularity.  We refer to this as \emph{continuous stable regularity} for various reasons to be explained later. In any case, we say that a function $f:V\times W \to [0,1]$ is  \textbf{$(k,\delta)$-stable} if there do not exist $a_{1},\ldots,a_{k}\in V$ and $b_{1},\ldots,b_{k}\in W$ such that $|f(a_i,b_j)-f(a_j,b_i)|\geq\delta$ for all $i<j$.  (This  generalizes the previous notion for graphs after possibly changing $k$; see Remark \ref{rem:stable-graphs}.) In Section \ref{sec:stable}, we will describe several families of examples of stable functions, drawing from the model-theoretic study of Hilbert spaces (see Corollary \ref{cor:hilbert}). 

Among our main results is Theorem \ref{thm:main-finite-thm}, which we quote now.

\begin{alphatheorem}\label{prethm:main-finite-thm}
Let $V$ and $W$ be finite sets, and suppose $f\colon V\times W\to [0,1]$ is a $(k,\delta)$-stable function. Then for any $\epsilon>0$ and any ``decay" function $\sigma\colon \N\to (0,1)$,  there are partitions $V=V_0\cup V_1\cup\ldots\cup V_m$ and $W=W_0\cup W_1\cup\ldots\cup W_n$, with $m,n\leq O_{k,\delta,\epsilon,\sigma}(1)$, satisfying the following properties.
\begin{enumerate}[\hspace{5pt}$\ast$]
\item For all $(i,j)\in [m]\times [n]$, the pair $(V_i,W_j)$ is $(5\delta+\epsilon;\sigma(mn))$-homogeneous.
\item $|V_0|\leq\epsilon|V_1|$ and $|W_0|\leq \epsilon|W_1|$.
\end{enumerate}
\end{alphatheorem}

Homogeneity is defined   in  Section \ref{sec:upstairs} (see Definition \ref{def:hom}).  Roughly speaking, the pair $(V_i,W_j)$ is $(\delta;\epsilon)$-homogenous (with respect to the fixed  $f\colon V\times W\to [0,1]$), if for all  but $\epsilon|W_{j}|$-many $b\in W_{j}$, for all but $\epsilon|V_{i}|$-many $a\in V_{i}$, the value $f(a,b)$ is within $\delta$ of a fixed number $r\in [0,1]$ that depends only on $i$ and $j$ (along with a dual statement quantifying in the reverse order). In analogy to homogeneity for graphs, this yields the ``density" bound $\|f|_{V_i\times W_j}-r\|_1\leq \delta+2\epsilon$ (see Remark \ref{rem:hompairs}). The proof of Theorem \ref{prethm:main-finite-thm} also provides  definability conditions on the sets $V_i$ and $W_j$ in terms of the function $f$ (see Remark \ref{rem:complexity}). Finally, we note that the bound $O_{k,\delta,\epsilon,\sigma}(1)$  is ineffective due to the use of pseudofinite methods.

The sets $V_0$ and $W_0$ in Theorem \ref{prethm:main-finite-thm} are small ``exceptional" sets of vertices, which can be incorporated into $V_1$ and $W_1$ to obtain a total partition (with certain costs; see Remark \ref{rem:ex-vert} and Theorem \ref{thm:main-finite2}). We also prove a version of Theorem \ref{prethm:main-finite-thm} involving balanced \emph{equipartitions} (see Theorem \ref{thm:equip}).    Moreover, Theorem \ref{thm:function-version} gives a strong decomposition theorem for stable functions in the sense of ``analytic" regularity. 

We now discuss our methods, as well as other aspects of this paper (including  purely model-theoretic results).  In combinatorics, the analytic generalization of regularity  is often viewed as being very useful, but only requiring routine variations of existing arguments. So we make the important remark that this is not the case for stable regularity. This is largely because of the shift from regularity to homogeneity, which yields a much stronger description of the object in question using highly structured ingredients. The adaptation of these ingredients to the setting of  functions will require several nontrivial steps and new results.

Szemer\'{e}di's original proof of his regularity lemma was direct and finitary. 
Likewise the proof of stable graph regularity in \cite{MaSh} was at the finitary level but with tools from stability such as the $2$-rank. 
On the other hand, it is well-known that  a theorem about all finite objects of a certain kind can be obtained from a theorem about a single ``nonstandard finite" (or  \emph{pseudofinite}) object (although this method typically does not yield effective bounds, where relevant).  Such an approach to graph regularity is described in a blog post by Tao \cite{TaoNSAblog}, with origins in broader work of Elek and Szegedy \cite{ElSz}. An account can also be found in course notes of the third author \cite{Pillay-pseudofinite}.  Insofar as stable  regularity is concerned, it is very natural to use pseudofinite methods as one can plug into to the existing theory of local stability in model theory.   This approach to stable graph regularity was carried out in \cite{MaPi}, giving a structure theorem for infinite bipartite graphs $(V,W;E)$ for which the formula $E(x,y)$ is stable, and in the presence of finitely additive (Keisler) measures  on the relevant Boolean algebras of subsets of $V$ and $W$.

This will be our approach to stable regularity in the current paper.  However, because we are working with a function $f:V\times W \to [0,1]$, rather than a relation $E\subseteq V\times W$, the relevant nonstandard environment is in the realm of {\em continuous logic} where the basic formulas are real-valued rather than Boolean-valued. The change in logic is one reason for our use of the expression ``continuous stable regularity". This also fits with combinatorics where the word ``continuous" is sometimes used to indicate the passage from graph relations to real-valued functions. Finally, we note that our setting is related to recent work of Chernikov and Towsner \cite{ChTowVCk} on tame regularity for $[0,1]$-valued $(k+1)$-ary functions of bounded VC$_k$-dimension (a higher arity analogue of NIP).  

There is some work on local (formula-by-formula) stability in continuous logic which we can and will appeal to (such as \cite{BYU}); but we will also need to develop  new results concerning Keisler measures and stable formulas in the continuous environment.  This is done in Section \ref{sec:measures}. In fact, we will work in large part under the weaker assumption that the given continuous formula 
$\phi(x,y)$ is ``$\delta$-stable" for a particular $\delta > 0$ (rather than  \emph{fully stable}, i.e.,  $\delta$-stable for all $\delta > 0$; see Definition \ref{def:stable2}).
For a given model $M$, the relevant Keisler measures will be regular Borel (probability) measures on the space $S_{\phi}(M)$ of complete $\phi$-types over $M$.

In the classical first-order context, a basic theorem is that if $\phi(x,y)$ is stable then any Keisler measure $\mu$ on the 
space $S_{\phi}(M)$ is a countable weighted sum $\mu = \sum\alpha_{i}p_{i}$ of Dirac measures,  where $p_i\in S_\varphi(M)$ and $\sum\alpha_{i} = 1$. An account of this result is given  by the third author in \cite{PiDR}, drawing from earlier  work of Keisler \cite{Keis}. We prove the following analogue in the continuous setting (quoting Theorem \ref{thm:sumtypes}).

\begin{alphatheorem}\label{prethm:sumtypes}
Let $M$ be a metric structure. Suppose $\varphi(x,y)$ is $\delta$-stable,  and $\mu$ is a Keisler measure on $S_{\phi}(M)$. Then there is a countable collection $\{C_{i}\}_{i\in I}$ of pairwise disjoint closed subsets of $S_{\phi}(M)$ each of 
diameter at most $2\delta$, such that  $\mu = \sum_{i\in I}\alpha_{i}\mu_{i}$, where $\mu_{i}$ is a Keisler measure concentrating on $C_{i}$ and  $\sum_{i\in I}\alpha_{i} = 1$.  
\end{alphatheorem}

Here ``diameter" refers to the local ``$d$-metric" on $S_\varphi(M)$ (see Definition \ref{def:top}), which is simply the discrete metric in the classical setting. When $\varphi(x,y)$ is fully stable, Theorem \ref{prethm:sumtypes} is similar to a result of Ben Yaacov from \cite{BYtop} (see Remark \ref{rem:BYtop}). In Section \ref{sec:upstairs}, we will use Theorem \ref{prethm:sumtypes} to prove  
the following  model-theoretic  regularity statement (paraphrasing Theorem \ref{thm:main-upstairs}). 

\begin{alphatheorem}\label{prethm:main-upstairs}
Let $M$ be an $\omega$-saturated metric structure. Suppose $\varphi(x,y)$ is  $\delta$-stable, and  $\mu$ and $\nu$ are Keisler measures on $S_\varphi(M)$ and $S_{\varphi^*}(M)$, respectively. Then for any $\epsilon>0$, there are $m,n\geq 1$ such that for any $\gamma\in (0,1)$, there are partitions $M^x=A_0\cup A_1\cup\ldots\cup A_m$ and $M^y=B_0\cup B_1\cup\ldots\cup B_n$ satisfying the following properties.
\begin{enumerate}[\hspace{5pt}$\ast$] 
\item For all $(i,j)\in [m]\times [n]$, the pair $(A_i,B_j)$ is $(5\delta;\gamma)$-homogeneous for $\mu$ and $\nu$.
\item $\mu(A_0)\leq \epsilon \mu(A_1)$ and $\mu(B_0)\leq\epsilon\mu(B_1)$.
\end{enumerate}
Moreover, $A_i$ and $B_j$ satisfy specific  ``definability conditions" involving $\varphi(x,y)$.
\end{alphatheorem}

A more elaborate version of this theorem is Lemma \ref{lem:upstairs-hom}, which is stated in such a form that when $M$ is an  ultraproduct of (continuous) finite structures, then the data can be transferred to the finite to achieve Theorem \ref{prethm:main-finite-thm}.  The very delicate aspects of the transfer are carried out in Sections \ref{sec:transfer} and \ref{sec:main-finite-proof}. When specialized to the classical first-order setting, the proof structure of Theorem \ref{prethm:main-upstairs} becomes a kind of synthesis of the strategies from \cite{MaPi} and \cite{PiDR}, with some additional simplifications.

\subsection*{Acknowledgements} 
The authors were supported by NSF grants DMS-1665025, DMS-1665035, DMS-1760212, and DMS-2054271 (Pillay), and DMS-1855503 (Conant). Much of the work in this paper was carried out during the 2021 Thematic Program on Trends in Pure and Applied Model Theory at the Fields Institute, where Pillay was appointed as a Simons Distinguished Visitor. Research funds from this appointment   were also used to support Chavarria and Conant at Fields. Chavarria's Notre Dame stipend for Fall 2021 was provided through DMS-1665025.

We thank the Fields Institute for their hospitality. The second author would  also like to thank J. Hanson for helpful conversations. Special thanks are  due to C. Terry for a number of invaluable discussions, and for suggesting Definition \ref{def:approxstr}.

\section{Stable functions}\label{sec:stable}

Before stating the definition of stability for functions, let us set some general notation, which will be used throughout the paper.
\begin{enumerate}[\hspace{5pt}$\ast$]
\item Given an integer $n\geq 1$, let $[n]=\{1,\ldots,n\}$. 

\item  Given a  set $X$ and a subset $A\seq X$, we let $\boldsymbol{1}_A\colon X\to \{0,1\}$ denote the indicator function of $A$. When  $X$ is understood, we write $\boldsymbol{1}$ for $\boldsymbol{1}_X$. 

\end{enumerate}

\begin{definition}\label{def:stable}
Let $f\colon V\times W\to \R$ be a function, where $V$ and $W$ are arbitrary sets. Given a linear order $I$ and  some $\delta\geq 0$, we say that $f$ is \textbf{$(I,\delta)$-stable} if there do not exist $a_i\in V$ and $b_i\in W$, for $i\in I$, such that 
\[
|f(a_i,b_j)-f(a_j,b_i)|\geq\delta \text{ for all $i<j$ from $I$}.
\]
Given $k\geq 1$, we say that $f$ is \textbf{$(k,\delta)$-stable} if it is $([k],\delta)$-stable; and we say that $f$ is \textbf{$\delta$-stable} if it is $(\N,\delta)$-stable. (Here we use the standard orders on $\N$ and $[k]$.) 
\end{definition}

It is easy to check that $f\colon V\times W\to \R$ is $\delta$-stable if and only if it is $(I,\delta)$-stable for all infinite linear orders $I$.   We also note that in certain contexts, $f$ is $\delta$-stable if and only if  it is $(k,\delta)$-stable for some $k\geq 1$ (specifically, when $f$ has bounded image and $(V,W,f)$ is ``saturated" in the model-theoretic sense). 

Our definition of stability has been formulated in a particular way so as to agree with previous work in the model-theoretic setting (e.g., \cite{BYU}), and connect to the appearance of stability in broader mathematics (as we discuss below). However, the definition differs slightly from the usual definition of stability for (bipartite) graphs, and so we take a moment to clarify this.

\begin{remark}\label{rem:stable-graphs}
Let $V$ and $W$ be sets and fix a binary relation $E\seq V\times W$. Then $E$ is \textbf{$k$-stable} if there do not exist $a_1,\ldots,a_k\in V$ and $b_1,\ldots,b_k\in W$ such that $E(a_i,b_j)$ holds if and only if $i\leq j$ (i.e., $(V,W;E)$ omits $([k],[k];\leq)$ as an induced subgraph). While this does not necessarily coincide with $(k,\delta)$-stability of the indicator function $\boldsymbol{1}_E$ for some choice of $\delta$, note that if $\boldsymbol{1}_E$ is $(k,1)$-stable then it easily follows that $E$ is $k$-stable as defined above. (Moreover, $\boldsymbol{1}_E$ is $(k,1)$-stable if and only if it is $(k,\delta)$-stable for any $\delta>0$.) 
 Conversely, by a routine Ramsey argument, one can show that  if $E$ is $k$-stable then $\boldsymbol{1}_E$ is $(O_k(1),1)$-stable. We will comment on a continuous analogue of this situation in the appendix.
\end{remark}

Call a function $f\colon V\times W\to \R$ \textbf{stable} if it is $\delta$-stable for all $\delta>0$. This notion turns out to be quite pervasive in functional analysis. Firstly, it corresponds to Grothendieck's ``double-limit" condition from \cite{GroWAP},  used to characterize relatively weakly compact sets in the Banach space of bounded continuous functions on an arbitrary topological space. This connection has been used to provide  analytic proofs of several important theorems from stability theory \cite{BYGro,PilGro}.  Secondly, Krivine and Maurey \cite{KrMau} defined a Banach space $B$ to be \emph{stable} if the function $f(x,y) = \|x+y\|$ is stable when restricted to the unit ball $U$ in $B$. They observe that any $L^p$-space is stable for $1\leq p<\infty$, and their main result is that any infinite-dimensional stable Banach space contains $\ell^p$ for some $1\leq p<\infty$. 

Thirdly, we discuss Hilbert spaces.  Given a Hilbert space $H$, there is a natural interpretation of $H$ as a metric structure in an appropriate (multi-sorted) language, which includes the vector space structure and the inner product. In this case, the complete theory $T$ of $H$ is stable, meaning that \emph{every}  formula is stable  in every model of $T$ (see \cite[Section 15]{BBHU}). Thus any Hilbert space satisfies the Krivine-Maurey definition of stability. Another natural formula to consider is the inner product function $f(x,y)=\langle x,y\rangle$, which is stable in any Hilbert space  when restricted to the unit ball.  It is now understood that the stability of the inner product in Hilbert spaces  is largely responsible for the recurring phenomenon of stable formulas arising naturally in several previous settings. A well-known example is \cite[Proposition 2.25]{HruAG} from Hrushovski's breakthrough work on the structure of approximate groups. Previous related results include \cite[Lemma 6.1]{HruPFF}, \cite[Lemma 5.21]{HrPiGLF}, and \cite[Lemma 3.4]{KiPi}, each of which relates to the study of \emph{simple} theories. The proposition from \cite{HruAG} is also a key ingredient in the model-theoretic proof from \cite{PillayStarchenko} of Tao's algebraic  regularity lemma for definable sets in finite fields \cite{TaoFFRL}. 

We now take the opportunity to explain how stability of Hilbert spaces can be used to produce very general examples of $(k,\delta)$-stable functions (such as those underlying \cite[Proposition 2.25]{HruAG}). Fix a formula $\varphi(x,y)$ in the language of Hilbert spaces. Assume $\varphi(x,y)$ is  $[\nv 1,1]$-valued when restricted to the unit ball (e.g., $\langle x,y\rangle$ or $\frac{1}{2}\|x+y\|$, but $x$ and $y$ could also be tuples of variables). Now let $H$ be a Hilbert space with unit ball $U$. Given arbitrary sets $V$ and $W$, and functions $g\colon V\to U$ and $h\colon W\to U$, let $\varphi_{g,h}\colon V\times W\to [\nv 1,1]$ be defined by $\varphi_{g,h}(a,b)=\varphi(g(a),h(b))$. 

\begin{theorem}\label{thm:hilbert}
For any $\varphi(x,y)$ and $\delta>0$, there is some $k\geq 1$ such that for any $H$, $V$, $W$, $g$, and $h$ as above, $\varphi_{g,h}\colon V\times W\to [\nv 1,1]$ is $(k,\delta)$-stable.
\end{theorem}
\begin{proof}
It suffices to show that for any $\varphi(x,y)$ and $\delta>0$, there is some $k\geq 1$ such that $\varphi(x,y)$ is $(k,\delta)$-stable in any Hilbert space. But this follows from compactness and the fact that any completion of the theory of Hilbert spaces is stable. 
\end{proof}

Let us point out some specific cases of the previous theorem.

\begin{corollary}\label{cor:hilbert}
For any $\delta>0$ there is some $k\geq 1$ such that, in each case below, the function $f$ is $(k,\delta)$-stable. 
\begin{enumerate}[$(i)$]
\item Let $X$, $V$, and $W$ be finite sets, and fix arbitrary functions $g\colon X\times V\to [\nv 1,1]$ and $h\colon X\times W\to [\nv 1,1]$. Let $f\colon V\times W\to [\nv 1,1]$ be defined by 
\[
\textstyle f(a,b)=\frac{1}{|X|}\sum_{x\in X}g(x,a)h(x,b).
\]
\item Let $G$ be a locally compact group, and fix  continuous functions $g,h\colon G\to [\nv 1,1]$ with compact support $C$. Let $\mu$ be a left Haar measure on $G$, normalized so that $\mu(C\cup C\inv)\leq 1$. The convolution $g\ast h\colon G\to [\nv 1,1]$ is the function   
\[
(g\ast h)(x)=\int g(t)h(t\inv x)\, d\mu.
\]
Let $f\colon G\times G\to [\nv 1,1]$ be defined by $f(x,y)=(g\ast h)(xy)$. 
\item Let $M$ be a metric structure and let $\lf(x)$ be a Keisler functional over $M$ (see Definition \ref{def:KeisF}). Fix $[\nv 1,1]$-valued formulas $\psi_1(x,y)$ and $\psi_2(x,z)$. Let $f\colon M^y\times M^z\to [0,1]$ be defined by $f(a,b)=\lf(\psi_1(x,a)\psi_2(x,b))$. 
\end{enumerate}
\end{corollary}
\begin{proof}
In each case, we apply Theorem \ref{thm:hilbert} with $\varphi(x,y)=\langle x,y\rangle$. For case $(i)$, let $H=\R^X$ with the \emph{normalized} inner product. Then $f$ is $\varphi_{g',h'}$, where $g'\colon V\to U$ maps $a$ to $g(x,a)$ and $h'\colon W\to U$ maps $b$ to $h(x,b)$.  

For case $(ii)$, let $H=L^2(G,\mu)$ with inner product $\langle u,v\rangle=\int uv\,d\mu$. Let $g',h'\colon G\to U$ be defined by $g'(a)=g(a t)$ and $h'(b)=h(t\inv b)$. Then $f$ is $\varphi_{g',h'}$. 

For case $(iii)$, let $\mu$ be the Borel measure on $S_x(M)$ corresponding to $\lf$, and let $H=L^2(S_x(M),\mu)$ with inner product as in $(ii)$. Let $g\colon M^y\to U$ and $h\colon M^z\to U$ be defined by $g(a)=\psi_1(x,a)$ and $h(b)=\psi_2(x,b)$. Then $f$ is $\varphi_{g,h}$. 
\end{proof}
 
Linear functionals on metric structures will be discussed in more detail in Section \ref{sec:LF}. 
 Note that case $(iii)$ of Corollary \ref{cor:hilbert} generalizes case $(i)$. 
Moreover, if one applies  $(iii)$ to a \emph{classical} first-order structure $M$, and identifies Boolean formulas with their indicator functions, then $f(a,b)=\mu(\psi_1(x,a)\wedge\psi_2(x,b))$ where $\mu$ is a  Keisler measure on $\Def_x(M)$. If $M$ is sufficiently saturated and $\mu$ is invariant (over some small set), then stability of $f$ is \cite[Proposition 2.25]{HruAG}. This last point, namely the use of stability of the inner product in Hilbert spaces to deduce \cite[Proposition 2.25]{HruAG}, was communicated to the third author by Remi Jaoui, after seeing it in a course by Hrushovski in Paris.

\begin{remark}
In the context of Theorem \ref{thm:hilbert}, suppose $H$ is $L^2(X,\mu)$ for some probability space $(X,\mu)$, $\varphi(x,y)$ is the inner product, and $g$ and $h$ take values in the set of $v\in H$ such that $\|v\|_\infty=1$ (which is a subset of the unit ball). Then in this case, one can show $k=\exp^2(O(\delta\inv))$  by an elementary argument due to Tao \cite{TaoARLblog}. Note that this situation covers most of the examples in Corollary \ref{cor:hilbert}.
\end{remark}

\section{Preliminaries on continuous logic}

We assume familiarity with the foundations of continuous logic and continuous model theory. See \cite{BBHU} for an introduction to this subject.

\subsection{Basic notions}

Let $\cL$ be a continuous language. We work in the setting where formulas may take values in a bounded subset of $\R$. Specifically, each predicate symbol $P$ in $\cL$  comes with a distinguished closed bounded interval $I_P\seq\R$ (in addition to the modulus of uniform continuity $\Delta_P$), and the definition of an $\cL$-structure  includes the requirement that $P$ is $I_P$-valued.

Let $T$ be a complete $\cL$-theory, and fix an $\cL$-formula $\varphi(x,y)$. For simplicity, we assume that $\varphi(x,y)$ is $[0,1]$-valued (but this is not crucial). Throughout this section, we work with a fixed model $M\models T$.

Let $S_\varphi(M)$ denote the space of local $\varphi$-types over $M$.  Recall that a type $p\in S_\varphi(M)$ is uniquely determined by the function $b\mapsto \varphi(p,b)$ from $M^y$ to $[0,1]$. Moreover, $S_\varphi(M)$ is a compact Hausdorff space under the natural quotient topology inherited from $S_x(M)$. Further details can be found in \cite[Section 6]{BYU}.

\begin{definition}  $~$
\begin{enumerate}[$(1)$]
\item A \textbf{$\varphi$-formula over $M$} is a continuous function $\psi\colon S_\varphi(M)\to \R$.
\item A \textbf{$\varphi$-generated formula} is a uniformly continuous combination of $\varphi(x,y_i)$ for $i<\omega$, i.e., a formula of the form $\zeta(x,\ybar)=\alpha(\varphi(x,y_i)_{i<\omega}$), where $\alpha\colon [0,1]^\omega\to \R$ is a continuous function.
\end{enumerate}
\end{definition}

\begin{remark}\label{rem:formulas}
In \cite{BYU}, continuous functions on $S_\varphi(M)$ are referred to as \emph{$\varphi$-predicates} (over $M$), and the word ``formula" is reserved for the smaller class of syntactic or \emph{finitary} formulas. For our purposes, this distinction will not be significant since we will either be working with the general class of formulas as defined above, or with very specific families of finitary formulas. 
\end{remark}

The next result, which is part of \cite[Fact 6.4]{BYU}, says that $\varphi$-formulas (over $M$) coincide with ``instances" of  $\varphi$-generated formulas.   

\begin{fact}
A function $\psi\colon S_\varphi(M)\to \R$ is continuous if and only if there is a $\varphi$-generated formula $\zeta(x,\ybar)$ and some $\bbar\in M^{\ybar}$ such that $\psi(x)=\zeta(x,\bbar)$.
\end{fact}

Given a $\varphi$-formula $\psi(x)$ over $M$ and a set $B\seq\R$, define
\[
[\psi(x)\in B]\coloneqq \{p\in S_\varphi(M):\psi(p)\in B\}.
\]
 Given $\varphi$-formulas $\psi_1(x)$, $\psi_2(x)$ over $M$, and Borel sets $B_1,B_2\seq\R$, we let 
\[
[\psi_1(x)\in B_1\wedge \psi_2(x)\in B_2]\coloneqq [\psi_1(x)\in B_1]\cap [\psi_2(x)\in B_2],
\]
and similarly for $\vee$. 

\begin{definition}$~$
\begin{enumerate}[$(1)$]
\item A \textbf{sub-basic open set} in $S_\varphi(M)$ is a set of the form $[\varphi(x,b)\in U]$ where $b\in M^y$ and $U\seq\R$ is a bounded open interval. 
\item A \textbf{basic open set in $S_\varphi(M)$} is a finite intersection of sub-basic open sets. 
\item A subset of $S_\varphi(M)$ is \textbf{explicitly open} if it is a finite union of basic open sets.
\end{enumerate}
\end{definition}

\begin{fact}
The basic open sets in $S_\varphi(M)$ (as defined above) are a basis for the topology on $S_\varphi(M)$.
\end{fact} 

\begin{definition}
A \textbf{Keisler measure on $S_\varphi(M)$} is a regular Borel probability measure on $S_\varphi(M)$.
\end{definition}

Let $\psi(x)$ be a $\varphi$-formula over $M$, and fix $B\seq \R$. If $B$ is Borel (resp., open, closed, etc.) then $[\psi(x)\in B]$ is Borel (resp., open, closed, etc.). In this case, if $\mu$ is a Keisler measure on $S_\varphi(M)$, then we write $\mu(\psi(x)\in B)$ to denote $\mu([\psi(x)\in B])$.

We will also use the previous notation in the global setting where $\mu$ is a Keisler measure on $S_x(M)$ and $\psi(x)$ is an $\cL_M$-formula. In this case, we have the pushforward measure $\tilde{\mu}$ of $\mu$ to $S_{\varphi}(M)$ and, if $\psi(x)$ is a $\varphi$-formula and $B\seq \R$ is Borel, then $\mu(\psi(x)\in B)=\tilde{\mu}(\psi(x)\in B)$.

\begin{definition}\label{def:alphaU}
Given a set $U\seq\R$, let $\alpha_U\colon \R\to [0,1]$ be defined by 
\[
\alpha_U(x)=\min\{d(x,\R\backslash U),1\}.
\]
\end{definition}

Note that $\alpha_U$ is uniformly continuous and so, in particular, it is a logical connective. Moreover, if $U$ is open then $\alpha_U(x)>0$ if and only if $x\in U$.

Suppose $V\seq S_\varphi(M)$ is an explicitly open set. Then we can write
\[
V=\bigcup_{i=1}^m\bigcap_{j=1}^{n_i}[\varphi(x,b_{i,j})\in U_{i,j}]
\]
where each $b_{i,j}$ is from $M^y$ and each $U_{i,j}$ is a bounded open interval. Set $\ybar=(y_{i,j})$ and define the $\varphi$-generated formula
\[
\psi(x,\ybar)\coloneqq \max_{1\leq i\leq m}\min_{1\leq j\leq n_i}\alpha_{U_{i,j}}(\varphi(x,y_{i,j})).
\]
Then $\psi(x,\ybar)>0$ is logically equivalent to $\bigvee_{i=1}^m\bigwedge_{j=1}^{n_i}\varphi(x,y_{i,j})\in U_{i,j}$. Therefore $V=[\psi(x,\bbar)>0]$. A formula of the form $\psi(x,\bbar)$ is called an \textbf{explicit $\varphi$-formula over $M$}. So we have shown that $V\seq S_\varphi(M)$ is an explicitly open set if and only if it is of the form $[\psi(x,\bbar)>0]$ for some explicit $\varphi$-formula $\psi(x,\bbar)$ over $M$. 

\begin{remark}\label{rem:basic-open}
One can further assume that the intervals $U_{i,j}$ have rational endpoints. However, we will not make this assumption in general. 
\end{remark}

Note that if $\psi(x)$ is a $\varphi$-formula over $M$, and $D\seq \R$ is closed, then the expression $\psi(x)\in D$ is logically equivalent to the  \emph{$\cL$-condition} $\alpha_{\R\backslash D}(\psi(x))=0$. In particular, $\psi(x)\in D$ defines a zeroset in $M^x$.

\begin{definition}\label{def:zeroset}
An \textbf{explicit $\varphi$-zeroset} is a subset of $M^x$ in the lattice\footnote{Recall that a lattice of subsets of some set $X$ is a collection of subsets  closed under (finite) unions and intersections.} generated by zerosets defined by $\varphi(x,b)\in D$, where $b\in M^y$ and $D\seq\R$ is closed.
\end{definition}

For example, if $\psi(x)$ is an explicit $\varphi$-formula over $M$, and $\eta\in [0,1]$, then $\psi(x)\geq\eta$ defines an explicit $\varphi$-zeroset in $M^x$.

\subsection{Linear functionals}
Let $V$ be an ordered real vector space with a norm $\|{\cdot}\|$. A \textbf{linear functional} on $V$ is map $\lf\colon V\to \R$ such that $\lf(rv+sw)=r\lf(v)+s\lf(w)$ for all $v,w\in V$ and $r,s\in \R$. A linear functional $\lf$ is \textbf{positive} if $\lf(v)\geq 0$ for all $v\in V$ such that $v\geq 0_V$. A linear functional $\lf$ is \textbf{bounded} if there is some $c\geq 0$ such that $|\lf(v)|\leq c\|v\|$ for all $v\in V$. The \textbf{operator norm} $\|\lf\|$ of a bounded linear functional $\lf$ is the infimum of all $c\geq 0$ satisfying the previous condition.

Given a compact Hausdorff space $X$, the space $\cC(X,\R)$ of all continuous functions from $X$ to $\R$ is a real ordered normed vector space (in fact a Banach space) under the pointwise partial order and the uniform norm $\|f\|_\infty=\sup_{x\in X}|f(x)|$.

\begin{fact} \label{fact:easy}
If $X$ is a compact Hausdorff space and $\lf$ is a positive linear functional on $\cC(X,\R)$, then $\lf$ is bounded and $\|\lf\|=\lf(\boldsymbol{1})$. 
\end{fact}
\begin{proof}
This is a basic exercise, but we will include the proof for later reference. We first fix $\varphi\in \cC(X,\R)$, and show $|\lf(\varphi)|\leq \lf(\boldsymbol{1})\|\varphi\|_\infty$. Note that $|\varphi|\leq \|\varphi\|_\infty\boldsymbol{1}$, and so $\|\varphi\|_\infty\boldsymbol{1}-\varphi\geq 0$ and $\|\varphi\|_\infty\boldsymbol{1}+\varphi\geq 0$. Thus $\|\varphi\|_\infty\lf(\boldsymbol{1})-\lf(\varphi)\geq 0$ and $\|\varphi\|_\infty \lf(\boldsymbol{1})+\lf(\varphi)\geq 0$, as desired. This shows that $\lf$ is bounded and $\|\lf\|\leq \lf(\boldsymbol{1})$. Conversely, $\lf(\boldsymbol{1})=|\lf(\boldsymbol{1})|\leq\|\lf\|\|\boldsymbol{1}\|_\infty=\|\lf\|$.
\end{proof}

Recall that if $X$ is a compact Hausdorff space, then a \textbf{Radon measure (on $X$)} is a  regular Borel measure $\mu$ on $X$ such that $\mu(X)<\infty$. Given a Radon measure $\mu$ on a compact Hausdorff space $X$, one obtains a positive linear functional $\lf_\mu$ on $\cC(X,\R)$ such that $\lf_\mu(f)=\int_X f\,d\mu$. Note that $\|\lf_\mu\|=\mu(X)$. The \emph{Riesz-Markov-Kakutani Theorem} states that the map $\mu\mapsto \lf_\mu$ is a bijection between Radon measures on $X$ and positive linear functionals on $\cC(X,\R)$. Therefore  regular Borel \emph{probability} measures correspond  to positive linear functionals of operator norm $1$.

\subsection{Linear functionals on metric structures}\label{sec:LF}

Let $\cL$ be a continuous language.
Given an $\cL$-structure $M$ and some sort $x$, we work with the vector space $\cC(S_x(M),\R)$ of \emph{$\cL$-formulas over $M$ in $x$}, with the pointwise partial order and uniform norm as in the previous subsection. We may also decorate the norm as $\|\varphi\|^M_\infty$ for emphasis. Note that the set of finitary (syntactic) $\cL$-formulas in $x$ over $M$ forms a subspace of $\cC(S_x(M),\R)$, which is dense by \cite[Fact 6.4]{BYU}.  Moreover, if $\varphi(x)$ is such a formula then $\|\varphi\|^M_\infty=\sup_{a\in M^x}|\varphi(a)|$.
 
 \begin{remark}
 Suppose $T$ is a complete $\cL$-theory. Then the ordered normed vector space structure on $\cL$-formulas (over $\emptyset$) in some fixed sort $x$ is part of the theory of $T$. In other words, if $M\equiv N$ then for any $\cL$-formula $\varphi(x)$, we have $\|\varphi\|^M_\infty=\|\varphi\|^N_\infty$, and given $\cL$-formulas $\varphi(x)$ and $\psi(x)$, we have $\varphi^M\leq\psi^M$ if and only if $\varphi^N\leq\psi^N$.
 \end{remark}

\begin{definition}\label{def:KeisF}
Given an $\cL$-structure $M$, a  \textbf{Keisler functional (in $x$) over $M$} is a positive linear functional $\lf$ on $\cC(S_x(M),\R)$ such that $\|\lf\|=1$. 
\end{definition}

 Note that any nonzero positive linear functional on $\cC(S_x(M),\R)$ can be normalized to a Keisler functional. Let us restate the Reisz-Markov-Kakutani-Theorem in this context.

 \begin{fact}
 If $M$ is an $\cL$-structure, then the map $\mu\mapsto \int d\mu$ is a bijection between Keisler measures on $S_x(M)$ and Keisler functionals in $x$ over $M$.
 \end{fact}
 
 We also recall that a Keisler functional is determined entirely by its behavior on the subspace of finitary formulas.

 \begin{proposition}\label{prop:lfcC}
Suppose $M$ is an $\cL$-structure, and $\lf_0$ is a positive  linear functional on finitary $\cL$-formulas over $M$ in  $x$. Then $\lf_0$ extends uniquely to a positive  linear functional $\lf$ on $\cC(S_x(M),\R)$ with $\|\lf\|=\|\lf_0\|$.
\end{proposition}
\begin{proof}
Using the same steps as in the proof of Fact \ref{fact:easy}, one sees that $\lf_0$ is bounded, and hence continuous. Thus the claim follows from the fact that the linear subspace of finitary $\cL$-formulas is dense in $\cC(S_x(M),\R)$ (see \cite[Fact 6.4]{BYU}), together with basic facts in functional analysis (see \cite[Exercise I.I.19]{Rudin-fun}). 
\end{proof}

In light of the previous result, we will sometimes view linear functionals as maps on $\cC(S_x(M),\R)$, while in other cases as maps on the subspace of finitary $\cL$-formulas over $M$, depending on the relevant context. 

We end this section by examining ultraproducts of functionals.
Fix an infinite index set $\Sigma$, a collection $(M_s)_{s\in\Sigma}$ of $\cL$-structures, and an ultrafilter $\cU$ on $\Sigma$. Let $M$ be the metric ultraproduct $\prod_{\cU} M_s$. Let $x$ be a fixed sort, and suppose that for all $s\in\Sigma$, we have a positive linear functional $\lf_s$ on $\cC(S_x(M_s),\R)$. Assume further that the set $\{\|\lf_s\|:s\in \Sigma\}$ of norms is bounded. Given a finitary $\cL$-formula $\varphi(x,b)$ over $M$, define $\lf(\varphi(x,b))=\lim_{\cU}\lf_s(\varphi(x,b^s))$ where $(b^s)_{s\in\Sigma}$ is a representative of $b$. 

\begin{proposition}
$\lf$ induces a well-defined positive linear functional on $\cC(S_x(M),\R)$. Moreover, $\|\lf\|=\lim_{\cU}\|\lf_s\|$. 
\end{proposition}
\begin{proof}
We first show $\lf$ is well-defined. Fix a finitary $\cL$-formula $\varphi(x,y)$ and metrically $\cU$-equivalent sequences $(b^s)_{s\in\Sigma}$ and $(c^s)_{s\in\Sigma}$. We verify  $\lim_{\cU}\lf_s(\varphi(x,b^s))=\lim_{\cU}\lf_s(\varphi(x,c^s))$. Toward this end, we fix $\epsilon>0$ and show that the set 
\[
X:=\{s\in\Sigma:|\lf_s(\varphi(x,b^s))-\lf_s(\varphi(x,c^s))|\leq\epsilon\}
\]
 is in $\cU$. Let $R>0$ be a bound on $\|\lf_s\|$ for all $s\in\Sigma$. By assumption, the set 
 \[
 Y:=\{s\in\Sigma:d(b^s,c^s)<\Delta_{\varphi(x,y)}(\epsilon/R)\}
 \]
 is in $\cU$. Moreover, if $s\in Y$ then $\|\varphi(x,b^s)-\varphi(x,c^s)\|_\infty\leq\epsilon/R$, and so 
\[
|\lf_s(\varphi(x,b^s))-\lf_s(\varphi(x,c^s))|=|\lf_s(\varphi(x,b^s)-\varphi(x,c^s))|\leq R\|\varphi(x,b^s)-\varphi(x,c^s)\|_\infty\leq\epsilon.
\]
Therefore $Y\seq X$, whence $X\in\cU$.

Now, since ultralimits preserve vector space operations, it follows that $\lf$ is a linear functional. We show next that $\lf$ is positive.  Fix an $\cL$-formula $\varphi(x,b)$ over $M$ such that $\varphi(x,b)\geq 0$. By {\L}o\'{s}'s Theorem, $\lim_{\cU}\inf_x\varphi(x,b^s)\geq 0$.  Now fix $\epsilon>0$. Then the set $Z:=\{s\in\Sigma:\varphi(x,b^s)\geq \nv\epsilon\}$ is in $\cU$. If $s\in Z$ then $\varphi(x,b^s)+\epsilon\geq 0$, and so $\lf_s(\varphi(x,b^s)+\epsilon)\geq 0$, which implies
\[
\lf_s(\varphi(x,b^s))\geq\nv\lf_s(\epsilon)=\nv\epsilon\lf_s(\boldsymbol{1})=\nv\epsilon\|\lf_s\|.
\]
Therefore, $\lf(\varphi(x,b))=\lim_{\cU}\lf_s(\varphi(x,b^s))\geq\nv\epsilon\lim_{\cU}\|\lf_s\|\geq \nv\epsilon R$. Since $\epsilon>0$ was arbitrary and $R$ is fixed, it follows that $\lf(\varphi(x,b))\geq 0$. So we have shown that $\lf$ is positive. Finally, note that $\|\lf\|=\lf(\boldsymbol{1})=\lim_{\cU}\lf_s(\boldsymbol{1})=\lim_{\cU}\|\lf_s\|$. 
\end{proof}

\begin{example}
Working in the above setting, suppose also that each $M_s$ is finite. Then for a given sort $x$, we can define the ``average value functional" $\lf_s$ such that if $\varphi(x)$ is a finitary $L$-formula over $M_s$ then $\lf_s(\varphi)=\frac{1}{|M_s^x|}\sum_{a\in M_s^x}\varphi(a)$. Note that each $\lf_s$ is a Keisler functional in $x$ over $M_s$. By the previous proposition, it follows that $\lf:=\lim_{\cU}\lf_s$ is a Keisler functional on $M$, which we refer to as the \textbf{pseudofinite average value functional on $M$ in sort $x$}.   
\end{example}

We will also need the following standard fact (see \cite[Proposition 7.6]{BBHU}).

\begin{fact}\label{fact:ultrasat}
Assume $\cL$ is countable and let $\{M_s:s\in\Sigma\}$ be a countable family of $\cL$-structures. Then $\prod_{\cU}M_s$ is $\omega_1$-saturated for any nonprincipal ultrafilter $\cU$ on $\Sigma$.
\end{fact}

\section{Keisler measures on stable formulas}\label{sec:measures}

Let $T$ be a complete $\cL$-theory, and fix a $[0,1]$-valued $\cL$-formula $\varphi(x,y)$. Since we will be working locally around $\varphi(x,y)$, there is no harm in assuming $\cL$ is countable. 

Given $r,s\in\R$, and some $\epsilon>0$, we write $r\approx_\epsilon s$ to denote $|r-s|\leq\epsilon$. We also let $r\dotminus s=\max\{r-s,0\}$. So $r\dotminus s=0$ if and only if $r\leq s$.

\begin{definition}\label{def:stable2}
Given $\delta\in [0,1]$, we say that $\varphi(x,y)$ is \textbf{$\delta$-stable (in $T$)} if for every $M\models T$, the function $\varphi\colon M^x\times M^y\to[0,1]$ is $\delta$-stable (as defined in Section \ref{sec:stable}). 

We say that $\varphi(x,y)$ is \textbf{stable} if it is $\delta$-stable for all $\delta>0$.
\end{definition}

It is not hard to show that when checking $\delta$-stability  of $\varphi(x,y)$ with respect to $T$, it suffices to consider a single $\omega$-saturated  model $M\models T$. In the next lemma, we further note that $\delta$-stability is an  ``open condition".

\begin{lemma}\label{lem:stableopen}
There is some $\delta_\varphi\in [0,1]$ such that, for any $\delta\in [0,1]$, $\varphi(x,y)$ is $\delta$-stable if and only if $\delta>\delta_\varphi$.
\end{lemma}
\begin{proof}
Suppose $\varphi(x,y)$ is $\delta$-stable for some $\delta>0$. We find some $\epsilon>0$ such that $\varphi(x,y)$ is $\delta'$-stable  for all $\delta'\in(\delta-\epsilon,\delta)$. This suffices to prove the lemma (take $\delta_\varphi=1$ if $\varphi(x,y)$ is not $\delta$-stable for any $\delta\in[0,1]$). By compactness, there is some $k\geq 1$ such that $\varphi^M$ is $(k,\delta)$-stable for all $M\models T$. Define the formula
\[
\theta(x_1,\ldots,x_k,y_1,\ldots,y_k)\coloneqq\max_{i<j}(\delta\dotminus |\varphi(x_i,y_j)-\varphi(x_j,y_i)|).
\]
Then for any $M\models T$ and $a_1,\ldots,a_k\in M^x$, $b_1\ldots,b_k\in M^y$, we have $\theta(\abar,\bbar)>0$. By compactness, there is some $\epsilon>0$ such that $\inf^M_{\xbar,\ybar}\theta(\xbar,\ybar)=\epsilon$ for all $M\models T$. Unpacking this, it follows immediately that $\varphi^M$ is $(k,\delta')$-stable for any $M\models T$ and $\delta'\in (\delta-\epsilon,\delta)$. Therefore $\varphi(x,y)$ is $\delta'$-stable for any $\delta'\in (\delta-\epsilon,\delta)$.
\end{proof}

Let $\varphi^*(y,x)$ denote the same formula $\varphi(x,y)$, but with the roles of object and parameter variables exchanged.

\begin{definition}
A \textbf{min-max $\varphi^*$-generated formula} is a formula of the form
\[
\zeta(y,\xbar)=\min_{1\leq i\leq m}\max_{1\leq j\leq n_i}\varphi(x_{i,j},y)
\]
for some $m,n_1,\ldots,n_m\geq 1$.
\end{definition}

 The next result  says that if $\varphi(x,y)$ is $\delta$-stable, then $\varphi$-types are ``uniformly approximately  $\varphi^*$-definable". See \cite[Lemma 7.4]{BYU} for details. 

\begin{lemma}\label{lem:UDT}
Assume $\varphi(x,y)$ is $\delta$-stable. Then there is a min-max $\varphi^*$-generated formula $\zeta^\delta_\varphi(y,\xbar)$ such that, for any $M\models T$ and $p\in S_\varphi(M)$, there is some $\cbar\in M^{\xbar}$ such that for all $b\in M^y$,  $\varphi(p,b)\approx_\delta \zeta^\delta_\varphi(b,\cbar)$.
\end{lemma}

It is easy to see that if $\varphi(x,y)$ is $\delta$-stable, then so is $\varphi^*(y,x)$, and thus by the previous lemma we obtain a min-max $\varphi$-generated formula $\zeta^{\delta}_{\varphi^*}(x,\ybar)$ satisfying the analogous conclusion for all types $q\in S_{\varphi^*}(M)$.

For the rest of this section, we fix a model $M\models T$. We now recall the topometric space structure on $S_\varphi(M)$, as well as the associated Cantor-Bendixson ranks, as defined in \cite{BYtop,BYU}. Recall that a \emph{topometric space} is a pair $(X,d)$ where $X$ is a  Hausdorff space and $d$ is a metric on $X$ satisfying the following properties.
\begin{enumerate}[$(i)$]
\item The metric refines the topology, i.e., for every open $V\seq X$ and every $p\in V$, there is some $\epsilon>0$ such that $\{q\in X:d(p,q)<\epsilon\}\seq V$.
\item The metric function  is lower semi-continuous, i.e., for any $\epsilon>0$, the set $\{(p,q)\in X\times X:d(p,q)\leq\epsilon\}$ is closed in $X\times X$.
\end{enumerate}

\begin{definition}\label{def:top}
The \textbf{$d$-metric} on $S_\varphi(M)$ is $d(p,q)=\sup_{b\in M^y}|\varphi(p,b)-\varphi(q,b)|$.
\end{definition}

Despite our use of the terminology ``$d$-metric", we note that this metric on $S_\varphi(M)$ can be quite different from what is usually called the $d$-metric on the space of \emph{complete} types (see \cite[Section 4.3]{BYU}). For example, if the metric on $M$ is discrete then so is the $d$-metric on complete types, although what we call the $d$-metric on $S_\varphi(M)$ may not be discrete (for a given $\varphi$).

\begin{fact}\label{fact:topo}
$(S_\varphi(M),d)$ is a (compact) topometric space.
\end{fact}
\begin{proof}
See \cite[Fact 6.2]{BYU} and \cite{BYtop} (especially Definition 1.2 and the remarks after Lemma 1.9). Given the definition of the metric on $S_\varphi(M)$, properties $(i)$ and $(ii)$ are also easy to check directly. 
\end{proof}

The following is a straightforward exercise that we will use later.

\begin{proposition}\label{prop:lip}
Suppose $\zeta(y,\xbar)$ is a min-max $\varphi^*$-generated formula. Then $|\zeta(q,\cbar)-\zeta(q',\cbar)|\leq d(q,q')$ for any $\cbar\in M^{\xbar}$ and $q,q'\in S_{\varphi^*}(M)$.
\end{proposition}

Given $C\seq S_\varphi(M)$, let $\diam(C)=\sup\{d(p,q):p,q\in C\}$.

\begin{definition}\label{def:CB}
Fix $\delta>0$.
\begin{enumerate}[$(1)$]
\item We define closed sets $X_{\delta,\alpha}\seq S_\varphi(M)$, for $\alpha$ an ordinal. Let $X_{\delta,0}=S_\varphi(M)$ and, for a limit ordinal $\alpha$, set $X_{\delta,\alpha}=\bigcap_{\beta<\alpha}X_{\delta,\beta}$. Finally, set
\[
X_{\delta,\alpha+1}=\bigcap\{F\seq X_{\delta,\alpha}:\text{$F$ is closed and $\diam(X_{\delta,\alpha}\backslash F)\leq\delta$}\}.
\]
\item Fix  a nonempty closed set $C\seq X$. The \textbf{$\delta$-Cantor-Bendixson rank of $C$} is 
\[
\CB_\delta(C)\coloneqq \sup\{\alpha:X_{\delta,\alpha}\cap C\neq\emptyset\}\in \text{Ord}\cup\{\infty\}.
\]
If $\CB_\delta(C)=\alpha<\infty$, then define $\CBm_\delta(C)\coloneqq C\cap X_{\delta,\alpha}$.
\end{enumerate}
\end{definition}

The previous definition also makes sense when $\delta=0$, and yields the usual Cantor-Bendixson rank of a topological space. However, we will not use this case.

\begin{proposition}\label{prop:BYU}
Fix $\delta>0$ and suppose $C\seq S_\varphi(M)$ is closed and nonempty.
\begin{enumerate}[$(a)$]
\item $\CBm_\delta(C)$ is closed and nonempty. 
\item If $\CB_\delta(C)<\infty$, and $D\seq C$ is a nonempty closed set disjoint from $\CBm_\delta(C)$, then  $\CB_\delta(D)<\CB_\delta(C)$.
\item If $\CB_\delta(C)<\infty$ then $\CBm_\delta(C)$ admits a finite open cover $U_1,\ldots,U_n$ such that $\diam(U_i\cap \CBm_\delta(C))\leq\delta$ for all $1\leq i\leq n$.
\end{enumerate}
\end{proposition}
\begin{proof}
These follow easily from the definitions. See also the remarks following \cite[Definition 7.9]{BYU}. 
\end{proof}

The following is a ``$\delta$-local" analogue of \cite[Proposition 7.11]{BYU}. 

\begin{proposition}\label{prop:CBstable}
If $\varphi(x,y)$ is $\delta$-stable then  $\CB_{2\delta}(S_\varphi(M))<\infty$.
\end{proposition}
\begin{proof}
This is more or less implicit in \cite{BYU}. But we will provide a sketch. Toward a contradiction, suppose $\CB_{2\delta}(S_\varphi(M))=\infty$. Arguing inductively as in \cite[Proposition 7.11]{BYU}, we may construct $M_0\preceq M$ of countable density character, and types $\{p_\eta:\eta\in 2^\omega\}$ in $S_\varphi(M_0)$ such that $d(p_\eta,p_\lambda)>2\delta$ for all distinct $\eta,\lambda\in 2^\omega$. By Lemma \ref{lem:stableopen}, we can fix  $\delta'<\delta$ such that $\varphi(x,y)$ is $\delta'$-stable. 
For each $\eta\in 2^\omega$,   apply Lemma \ref{lem:UDT} to find $\cbar_\eta\in M_0^{\xbar}$ such that for all $b\in M_0^y$, $\varphi(p_\eta,b)\approx_{\delta'}\zeta^{\delta'}_\varphi(b,\cbar_\eta)$.  Since $M_0^{\xbar}$ has countable density character,  there are distinct $\eta,\lambda\in 2^\omega$ such that $\|\zeta^{\delta'}_\varphi(y,\cbar_\eta)-\zeta^{\delta'}_\varphi(y,\cbar_\lambda)\|_\infty<2(\delta-\delta')$. But then $d(p_\eta,p_\lambda)<2\delta$ by the triangle inequality, which is a contradiction. 
\end{proof}

Our next main goal is a description of local Keisler measures for $\delta$-stable formulas. We will obtain a continuous analogue of the well-known result that in a classical discrete theory, if $\varphi(x,y)$ is stable then any Keisler measure on $S_\varphi(M)$ can be written as an infinite weighted sum of types (see \cite[Fact 1.1]{PiDR}).

\begin{lemma}\label{lem:sumtypes}
Let $\mu$ be a Keisler measure on $S_\varphi(M)$, and fix $\delta>0$. Then for any closed set $C\seq S_\varphi(M)$, if $\mu(C)>0$ and $\CB_\delta(C)<\infty$, then there is a closed set $C'\seq C$ such that $\diam(C')\leq\delta$ and $\mu(C')>0$.
\end{lemma}
\begin{proof}
 We first prove the lemma in the special case that  $\mu(\CBm_\delta(C))>0$. So assume this is the case, and  let $D=\CBm_\delta(C)$. By Proposition \ref{prop:BYU}$(c)$, there are open sets $U_1,\ldots,U_n\seq S_\varphi(M)$ such that $D\seq\bigcup_{i=1}^n U_i$ and $\diam(U_i\cap D)\leq\delta$. Let $Y_i=U_i\cap D$, and note that each $Y_i$ is Borel. Since $\mu(D)>0$, there must be some $1\leq i\leq n$ such that $\mu(Y_i)>0$. Let $C'=\overline{Y_i}$ and note that $C'\seq D\seq C$ and $\mu(C')\geq\mu(Y_i)>0$. Finally, by lower semi-continuity of the metric, and since $\diam(Y_i)\leq\delta$, it follows that $\diam(C')\leq\delta$.

Now we prove the lemma in the general case. Let $C\seq S_\varphi(M)$ be closed with $\mu(C)>0$ and $ \CB_\delta(C)<\infty$. We proceed by induction on $\CB_\delta$-rank. If $\CB_\delta(C)=0$ then $\CBm_\delta(C)=C$, and so we can apply the special case above. So assume the result for ranks strictly less than $\CB_\delta(C)$.  

 Let $C_1=\CBm_\delta(C)$. By the special case above, we can assume $\mu(C_1)=0$. Set $X=C\backslash C_1$. Then $X$ is Borel and $\mu(X)=\mu(C)>0$. By regularity, there is a closed set $D\seq X$ such that $\mu(D)>0$. By Proposition \ref{prop:BYU}$(b)$, we have $\CB_\delta(D)<\CB_\delta(C)$. So we can apply the induction hypothesis to find a closed set $C'\seq D\seq C$ such that $\diam(C')\leq\delta$ and $\mu(C')>0$, as desired. 
\end{proof}

Given a Keisler measure $\mu$ on $S_\varphi(M)$ and a Borel set $X\seq S_\varphi(M)$, with $\mu(X)>0$, the \emph{localization of $\mu$ at $X$} is the Keisler measure $\mu_X$ on $S_\varphi(M)$ such that $\mu_X(B)=\mu(B\cap X)/\mu(X)$ for any Borel $B\seq S_\varphi(M)$. We now restate and prove Theorem \ref{prethm:sumtypes} from the introduction.

\begin{theorem}\label{thm:sumtypes}
Assume $\varphi(x,y)$ is $\delta$-stable, and let $\mu$ be a Keisler measure on $S_\varphi(M)$. Then there is a countable  collection $(C_i)_{i\in I}$ of pairwise disjoint closed sets in $S_\varphi(M)$ such that:
\begin{enumerate}[$(i)$]
\item  for all $i\in I$, $\diam(C_i)\leq 2\delta$ and $\mu(C_i)>0$, and 
\item $\mu=\sum_{i\in I}\alpha_i\mu_{C_i}$, where $\alpha_i=\mu(C_i)$.
\end{enumerate}
\end{theorem}
\begin{proof}
Let $\{C_i:i\in I\}$ be a maximal family of pairwise disjoint closed sets in $S_\varphi(M)$ with positive measure and diameter at most $2\delta$. Then $I$ must be countable by countable additivity of $\mu$. Let $B=\bigcup_{i\in I}C_i$, and note that $B$ is Borel. We claim that $\mu(B)=1$. Indeed, if not then by regularity there is a closed set $C\seq S_\varphi(M)$ such that $\mu(C)>0$ and $C\cap B=\emptyset$. By Proposition \ref{prop:CBstable}, we may apply Lemma \ref{lem:sumtypes} to $C$ to obtain a closed set $C'\seq C$ with $\mu(C')>0$ and $\diam(C')\leq 2\delta$. Then $C$ is disjoint from $C_i$ for all $i\in I$, contradicting maximality of the family.  

Now let $\alpha_i\coloneqq\mu(C_i)$. Then, for any Borel set $X\seq S_\varphi(M)$, we have
\[
\mu(X)=\mu(B\cap X)=\sum_{i\in I}\mu(C_i\cap X)=\sum_{i\in I}\alpha_i\mu_{C_i}(X).
\]
Therefore $\mu=\sum_{i\in I}\alpha_i\mu_{C_i}$, as desired. 
\end{proof}

\begin{remark}\label{rem:BYtop}
Call a subset of $S_\varphi(M)$ \emph{$\epsilon$-finite} if it can be covered by finitely many sets of diameter at most $\epsilon$. It follows from Theorem \ref{thm:sumtypes} that if $\varphi(x,y)$ is $\delta$-stable and $\mu$ is a Keisler measure on $S_\varphi(M)$, then for any $\epsilon>0$ there is some closed $2\delta$-finite set $C\seq S_\varphi(M)$ such that $\mu(C)>1-\epsilon$. This is closely related to \cite[Theorem 3.31]{BYtop} which (in our setting) says that if $\varphi(x,y)$ is stable then for any $\epsilon>0$ there is a metrically compact set $C\seq S_\varphi(M)$ such that $\mu(C)>1-\epsilon$. One can deduce this  from Theorem \ref{thm:sumtypes} by choosing a closed $\frac{1}{n}$-finite set $C_n$ such that $\mu(C_n)>1-2^{\nv n}\epsilon$, and setting $C=\bigcap_{n>0}C_n$.
\end{remark}

Before continuing with the main theme of this paper, we state some corollaries connecting Theorem \ref{thm:sumtypes} with the idea of approximating measures by types.

\begin{corollary}\label{cor:limit0}
Suppose $\varphi(x,y)$ is $\delta$-stable, and let $\mu$ be a Keisler measure on $S_\varphi(M)$. Then  there is a finitely supported Keisler measure $\mu'$ on $S_{\varphi}(M)$ such that  for each $b\in M^y$, 
\[
\int_{S_\varphi(M)}\varphi(x,b)\,d\mu\approx_{2\delta} \int_{S_\varphi(M)} \varphi(x,b)\, d\mu'.
\]
\end{corollary}
\begin{proof}
By Lemma \ref{lem:stableopen}, we can fix some $\delta'<\delta$ such that $\varphi(x,y)$ is $\delta'$-stable.
By Theorem \ref{thm:sumtypes}, we can write $\mu=\sum_{i\in I}\alpha_i\mu_{C_i}$ where $I$ is countable and the $C_i$'s are pairwise disjoint closed sets of diameter at most $2\delta'$. For each $i\in I$, pick some $p_i\in C_i$. So for any $q\in C_i$ and $b\in M^y$, we have $\varphi(q,b)\approx_{2\delta'} \varphi(p_i,b)$. It follows that for any $i\in I$ and any $b\in M^y$, 
\begin{equation*}
\int_{C_i}\varphi(x,b)\,d\mu_{C_i}\approx_{2\delta'} \int_{C_i}\varphi(p_i,b)\, d\mu_{C_i}=\varphi(p_i,b).\tag{$\dagger$}
\end{equation*}
Let $\epsilon=\delta-\delta'$, and choose a finite set $I_0\seq I$ such that $\alpha\coloneqq\sum_{i\in I_0}\alpha_i\geq 1-\epsilon$. Consider the finitely supported Keisler measure $\mu'=\alpha\inv\sum_{i\in I_0}\alpha_ip_i$, where $p_i$ is identified with its Dirac measure. Then by $(\dagger)$, we have that for any $b\in M^y$,
\begin{multline*}
\int_{S_\varphi(M)}\varphi(x,b)\,d\mu=\sum_{i\in I}\alpha_i\int_{C_i}\varphi(x,b)\,d\mu_{C_i}\approx_{2\delta'}\sum_{i\in I}\alpha_i\varphi(p_i,b)\\
\approx_\epsilon \sum_{i\in I_0}\alpha_i\varphi(p_i,b)=\alpha\int_{S_\varphi(M)}\varphi(x,b)\,d\mu'\approx_\epsilon \int_{S_\varphi(M)}\varphi(x,b)\,d\mu',
\end{multline*} 
as desired.
\end{proof}

When $\varphi(x,y)$ is  stable we obtain the following conclusion, which also follows from \cite[Corollary 3.32]{BYtop}.
 
\begin{corollary}\label{cor:limit2} 
Suppose $\phi(x,y)$ is stable, and let $\mu$ be a Keisler measure on $S_{\phi}(M)$. 
Then there is a sequence $(\mu_n)_{n=0}^\infty$ of  finitely supported Keisler measures on $S_\varphi(M)$ such that for all $\epsilon>0$ there is an $N$ such that for all $n\geq N$ and $b\in M^y$,  
\[
\int_{S_y(M)} \phi(x,b)\, d\mu   \approx_\epsilon \int_{S_y(M)}\phi(x,b)\, d\mu_{n} .
\] 
\end{corollary}

We point out that both of the previous corollaries can be rephrased as statements about  linear functionals.
 Let us also note that the conclusion of Corollary \ref{cor:limit2} is a continuous generalization of what is known in the classical first order case for NIP formulas, and thus we expect that it  extends to the setting where the continuous formula $\phi(x,y)$ is NIP.
But in any case the approximation by finitely supported measures in  Corollary \ref{cor:limit2}  is much weaker than the statement of Theorem \ref{thm:sumtypes}.  In the classical  first order case, this is analogous to  property $(b)$ below of a Keisler $\phi$-measure $\mu$ over $M$ being substantially stronger than property $(a)$ below.
\begin{enumerate}[$(a)$]
\item For each $\epsilon > 0$, there are $p_{1},\ldots,p_{k}\in S_{\phi}(M)$ such that for any $b\in M$, $\mu(\phi(x,b))$ is within $\epsilon$ of the average value of $\phi(x,b)$ at the $p_{i}$'s.
\item For each $\epsilon> 0$, there are $p_{1},\ldots,p_{k}\in S_{\phi}(M)$ such that $\mu(\{p_{1},\ldots,p_{k}\}) \geq 1-\epsilon$.
\end{enumerate}
In the discrete setting, $\varphi(x,y)$ is stable if and only if property $(b)$ holds for all Keisler $\varphi$-measures $\mu$; and $\varphi(x,y)$ is NIP if and only if property $(a)$ holds for all $\mu$ and all $\varphi$-generated formulas $\theta(x,\ybar)$. 

\section{Model-theoretic structure of stable formulas}\label{sec:upstairs}

The goal of this section is to prove Theorem \ref{prethm:main-upstairs}, which provides a structure theorem for stable continuous formulas  in terms of ``homogeneous pairs". We first define our notion of homogeneity precisely, and in a general setting. 

Let $V$ and $W$ be nonempty sets, and fix a function $f\colon V\times W\to [0,1]$. Let $\cA$ and $\cB$ be Boolean algebras of subsets of $V$ and $W$, respectively. Assume that $\cA$ contains any set of the form $\{x\in V:f(x,b)\in D\}$ where $b\in W$ is fixed and $D\seq [0,1]$ is a closed interval. Also assume the analogous condition for $\cB$. Let $\mu$ and $\nu$ be finitely additive probability measures on $\cA$ and $\cB$, respectively.

\begin{definition}\label{def:hom}
Fix real numbers $\delta,\gamma,\epsilon>0$. 
Given $V_*\in\cA$ and $W_*\in\cB$, we say that the pair $(V_*,W_*)$ is \textbf{$(\delta;\gamma,\epsilon)$-homogeneous for $\mu$ and $\nu$} if there are $V'\seq V_*$ and $W'\seq W_*$ satisfying the following properties:
\begin{enumerate}[$(i)$]
\item $V'\in\cA$ and $\mu(V')\geq (1-\gamma)\mu(V_*)$.
\item $W'\in\cB$ and $\nu(W')\geq (1-\gamma)\nu(W_*)$.
\item There is some $r\in [0,1]$ such that for all $b\in W'$,
\[
\mu(\{a\in V_*:f(a,b)\approx_\delta r\})\geq (1-\epsilon)\mu(V_*).
\]
\item There is some $s\in [0,1]$ such that for all $a\in V'$,
\[
\nu(\{b\in W_*:f(a,b)\approx_\delta s\})\geq (1-\epsilon)\nu(W_*).
\]
\end{enumerate}
If $\gamma=\epsilon$, then we say $(V_*,W_*)$ is \textbf{$(\delta;\gamma)$-homogeneous for $\mu$ and $\nu$}. 
\end{definition}

Note that the fixed function $f$ has been suppressed from the terminology. 
When the ambient measures are fixed, and there is no possibility of confusion, we will also often omit ``for $\mu$ and $\nu$". For example, if $V$ and $W$ are finite, then we will generally be in the situation where  $\mu$ and $\nu$ are the normalized counting measures on $\cP(V)$ and $\cP(W)$. Let us now comment more specifically on the finite case. Given a (nonempty) finite set $X$ and a function $f\colon X\to \R$, we define the \emph{normalized} $\ell^1$-norm $\|f\|_1=\frac{1}{|X|}\sum_{x\in X}|f(x)|$.

\begin{remark}\label{rem:hompairs}
Assume $V$ and $W$ are finite, and suppose $(V_*,W_*)$ is $(\delta;\gamma,\epsilon)$-homogeneous, witnessed by $r,s\in [0,1]$. Then: 
\begin{enumerate}[\hspace{5pt}$\ast$]
\item for all but at most $\gamma|W_*|$-many $b\in W_*$, for all but at most $\epsilon|V_*|$-many $a\in V_*$, we have $f(a,b)\approx_\delta r$, \emph{and dually},
\item  for all but at most $\gamma|V_*|$-many $a\in V_*$, for all but at most $\epsilon|W_*|$-many $b\in W_*$, we have $f(a,b)\approx_\delta s$.
\end{enumerate}
In this case, it should  intuitively follow that $r$ and $s$ are not much different. Indeed, if $f'=f|_{V_*\times W_*}$ then by a direct computation, $\|f'(x,y)-r\|_1$ and $\|f'(x,y)-s\|_1$ are both bounded by $\delta+\gamma+\epsilon$. So $|r-s|\leq 2(\delta+\gamma+\epsilon)$ by the triangle inequality.
\end{remark}

We now return to the previous setting  where $T$ is a complete continuous $\cL$-theory, with $\cL$ countable. Fix $M\models T$ and a $[0,1]$-valued $\cL$-formula $\varphi(x,y)$. The first step toward our main result is the following corollary of Theorem \ref{thm:sumtypes}. 

\begin{corollary}\label{cor:sumtypes}
Assume $\varphi(x,y)$ is $\delta$-stable, and let $\mu$ be a Keisler measure on $S_\varphi(M)$. Then for any $\epsilon>0$, there are pairwise disjoint closed sets $C_1,\ldots,C_n\seq S_\varphi(M)$ satisfying the following properties.
\begin{enumerate}[$(i)$]
\item $\diam(C_i)\leq 2\delta$ for all $1\leq i\leq n$.
\item $\mu(\bigcup_{i=1}^n C_i)>1-\epsilon\mu(C_1)$.
\item For any $\gamma\in (0,1)$, there are pairwise disjoint explicitly open sets $V_1,\ldots,V_n\seq S_\varphi(M)$ such that, for all $1\leq i\leq n$, $C_i\seq V_i$ and $\mu(V_i\backslash C_i)<\gamma\mu(V_i)$. 
\end{enumerate}
\end{corollary}
\begin{proof}
Let $\{C_i:i\in I\}$ be as in Theorem \ref{thm:sumtypes}, and set $\alpha_i\coloneqq\mu(C_i)$. Identify $I$ with an initial segment of $\Z^+$,  and choose $n\geq 1$ large enough so that $\sum_{i=1}^n\alpha_i>1-\epsilon\alpha_1$. This yields $C_1,\ldots,C_n$ satisfying  $(i)$ and $(ii)$. Toward  $(iii)$, fix $\gamma\in (0,1)$.

 For all distinct $i,j\leq n$, we can choose disjoint open sets $V_{i,j}$ and $W_{i,j}$ such that $C_i\seq V_{i,j}$ and $C_j\seq W_{i,j}$. Set $V'_i=\bigcap_{j\neq i}(V_{i,j}\cap W_{j,i})$.  Then $V'_i$ is open, $C_i\seq V'_i$, and if $i,j\leq n$ are distinct then $V'_i\cap V'_j\seq V_{i,j}\cap W_{i,j}=\emptyset$. 

Now, by regularity of $\mu$, there is an open set $U_i\supseteq C_i$ such that $\mu(U_i)<\alpha_i/(1-\gamma)$. Let $V''_i=V'_i\cap U_i$. Then $V''_i$ is open, $C_i\seq V''_i$, $\mu(V''_i)<\alpha_i/(1-\gamma)$, and $V''_1,\ldots,V''_n$ are pairwise disjoint.  Since $C_i$ is compact, there is an explicitly open set $V_i$ such that $C_i\seq V_i\seq V''_i$. Note that $V_1,\ldots,V_n$ are pairwise disjoint and, for all $1\leq i\leq n$, we have
\[
\mu(V_i\backslash C_i)=\mu(V_i)-\alpha_i<\gamma\mu(V_i),
\]
where the final inequality follows from $\mu(V_i)\leq\mu(V''_i)\leq\mu(U_i)<\alpha_i/(1-\gamma)$.
\end{proof}

We now prove a rather technical lemma that will be used to obtain  both Theorems \ref{prethm:main-finite-thm} and \ref{prethm:main-upstairs} from the introduction.

\begin{lemma}\label{lem:upstairs-hom}
Assume $\varphi(x,y)$ is $\delta$-stable, and let $\mu$ and $\nu$ be Keisler measures on $S_\varphi(M)$ and $S_{\varphi^*}(M)$, respectively. Fix some $\epsilon>0$. Then there are $m,n\geq 1$ such that for any $\gamma\in (0,1)$, there are:
\begin{enumerate}[\hspace{5pt}$\ast$]
\item explicit $\varphi$-formulas $\psi_1(x),\ldots,\psi_m(x)$ over $M$,
\item explicit $\varphi^*$-formulas $\theta_1(y),\ldots,\theta_n(y)$ over $M$,
\item finite tuples $\cbar_1,\ldots,\cbar_m,\dbar_1,\ldots,\dbar_n$ from $M$,  and
\item some $\eta>0$,
\end{enumerate}
satisfying the following properties.
\begin{enumerate}[$(1)$]
\item $[\psi_1(x)>0],\ldots,[\psi_m(x)>0]$ are pairwise disjoint.
\item $[\theta_1(y)>0],\ldots,[\theta_n(y)>0]$ are pairwise disjoint.
\item $\mu(\bigvee_{i=1}^m\psi_i(x)\geq\eta)~>~1-\epsilon\mu(\psi_1(x)\geq\eta)$.
\item $\nu(\bigvee_{j=1}^n \theta_j(y)\geq\eta)~>~1-\epsilon\nu(\theta_1(y)\geq\eta)$.
\item If $1\leq i\leq m$ then, for all $b\in M^y$,
\[
\mu\!\left(\psi_i(x)\geq\eta\wedge\varphi(x,b)\approx_{3\delta}\zeta^\delta_\varphi(b,\cbar_i)\right)~>~(1-\gamma)\mu(\psi_i(x)>0).
\]
\item If $1\leq j\leq n$ then, for all $a\in M^x$, 
\[
\nu\!\left(\theta_j(y)\geq\eta\wedge\varphi(a,y)\approx_{3\delta}\zeta^\delta_{\varphi^*}(a,\dbar_j)\right)~>~(1-\gamma)\nu(\theta_j(y)>0).
\]
\item For all $(i,j)\in [m]\times [n]$, there is some $r\in [0,1]$ such that 
\[
\nu\!\left(\theta_j(y)\geq\eta\wedge\zeta^\delta_\varphi(y,\cbar_i)\approx_{2\delta} r\right)~>~(1-\gamma)\nu(\theta_j(y)>0).
\]
\item For all $(i,j)\in [m]\times [n]$, there is some $s\in [0,1]$ such that 
\[
\mu\!\left(\psi_i(x)\geq\eta\wedge\zeta^\delta_{\varphi^*}(x,\dbar_j)\approx_{2\delta} s\right)~>~(1-\gamma)\mu(\psi_i(x)>0).
\]
\end{enumerate}
\end{lemma}

\begin{remark}\label{rem:upstairs-hom}
Before starting the proof of Lemma \ref{lem:upstairs-hom}, we take a moment to analyze properties $(5)$-$(8)$, and explain the connection to homogeneous pairs. 

Fix $(i,j)\in [m]\times[n]$, and consider the sets $V_i\coloneqq [\psi_i(x)\geq\eta]\seq S_\varphi(M)$ and $W_j\coloneqq [\theta_j(y)\geq\eta]\seq S_{\varphi^*}(M)$. Then $(5)$ says that for any $b\in M^y$, the function $\varphi(x,b)$ is within $3\delta$ of the \emph{fixed value} $\zeta^\delta_\varphi(b,\cbar_i)$ on almost all of $V_i$. For readers familiar with \cite{MaSh},  this is a functional analogue of the notion of a ``good set" in stable graph regularity. By itself, this  does not give homogeneity, since the value $\zeta^\delta_\varphi(b,\cbar_i)$ depends on the choice of $b$. However, property $(7)$ says that there is some single $r\in [0,1]$ (depending only on $i$ and $j$) such that the function $\zeta^\delta_\varphi(y,\cbar_i)$ is within $2\delta$ of $r$ on almost all of $W_j$. Applying the triangle inequality, we conclude that for almost all $b\in W_j(M)$, $\varphi(x,b)$ is within $5\delta$ of $r$ on almost all of $V_i$. By symmetric arguments with properties $(6)$ and $(8)$, we also have some $s\in [0,1]$ such that for almost all $a\in V_i(M)$, $\varphi(a,y)$ is within $5\delta$ of $s$ on almost all of $W_j$. Altogether, $(V_i,W_j)$  resembles a homogeneous pair. 
\end{remark}

\begin{proof}[\textnormal{\textbf{Proof of Lemma \ref{lem:upstairs-hom}}}]
Apply Corollary \ref{cor:sumtypes} to both $\mu$ and $\nu$ to obtain  pairwise disjoint closed sets $C_1,\ldots,C_m\seq S_\varphi(M)$ and pairwise disjoint closed sets $D_1,\ldots, D_n\seq S_{\varphi^*}(M)$ satisfying the following properties.
\begin{enumerate}[$(i)$]
\item For all $(i,j)\in [m]\times [n]$, $\diam(C_i)\leq 2\delta$ and $\diam(D_j)\leq 2\delta$.
\item $\mu(\bigcup_{i=1}^m C_i)>1-\epsilon\mu(C_1)$ and $\nu(\bigcup_{j=1}^n D_j)>1-\epsilon\nu(D_1)$. 
\item Corollary \ref{cor:sumtypes}$(iii)$ holds for for both $C_1,\ldots,C_m$ and $D_1,\ldots,D_n$. 
\end{enumerate}

Now fix $\gamma\in (0,1)$. By $(iii)$,  there are pairwise disjoint explicitly open sets $V_1,\ldots,V_m\seq S_\varphi(M)$, and pairwise disjoint explicitly opens sets $W_1,\ldots,W_n\seq S_{\varphi^*}(M)$, such that if $(i,j)\in [m]\times [n]$ then $C_i\seq V_i$, $D_j\seq W_j$, $\mu(V_i\backslash C_i)<\gamma\mu(V_i)$, and $\nu(W_j\backslash D_j)<\gamma\nu(D_j)$.

Let $\psi_i(x)$ be an explicitly open $\varphi$-formula over $M$ such that $V_i=[\psi_i(x)>0]$. Let $\theta_j(y)$ be an explicitly open $\varphi^*$-formula over $M$ such that $W_j=[\theta_j(y)>0]$. 
Note that properties $(1)$ and $(2)$ of the lemma hold. For the rest of the properties, we need to define the appropriate $\cbar_i$, $\dbar_j$, and $\eta>0$.

For each $1\leq i\leq m$, let $p_i$ be a fixed type in $C_i$ and, using Lemma \ref{lem:UDT}, choose a finite tuple $\cbar_i$ from $M$ such that for all $b\in M^y$, $\zeta^\delta_\varphi(b,\cbar_i)\approx_\delta \varphi(p_i,b)$. For each $1\leq j\leq n$, let $q_j$ be a fixed type in $D_j$, and let $\dbar_j$ be chosen similarly (using Lemma \ref{lem:UDT} applied to $\varphi^*(x,y)$ and   $\zeta^\delta_{\varphi^*}(x,\ybar))$.

Next, let $\tau>0$ be small enough so that $\mu(V_i\backslash C_i)<(\gamma-\tau)\mu(V_i)$ for all $1\leq i\leq m$, and $\nu(W_j\backslash D_j)<(\gamma-\tau)\nu(W_j)$ for all $1\leq j\leq n$. Note also that, by $(ii)$ above,
\[
\textstyle \mu(\bigcup_{i=1}^m V_i)>1-\epsilon\mu(V_1)\mand \nu(\bigcup_{j=1}^n W_j)>1-\epsilon\nu(W_1).
\]
Altogether, by countable additivity, we may choose some $\eta>0$ satisfying the following conditions:
\begin{enumerate}[\hspace{5pt}$\ast$]
\item $\mu(\psi_i(x)\geq\eta)\geq (1-\tau)\mu(V_i)$ for all $1\leq i\leq m$, 
\item $\nu(\theta_j(y)\geq\eta)\geq(1-\tau)\nu(W_j)$ for all $1\leq j\leq n$,
\item $\mu(\bigvee_{i=1}^m\psi_i(x)\geq\eta)>1-\epsilon\mu(\psi_1(x)\geq\eta)$, and 
\item $\nu(\bigvee_{j=1}^n \theta_j(y)\geq\eta)>1-\epsilon\nu(\theta_1(y)\geq\eta)$.
\end{enumerate}
In particular, we have properties $(3)$ and $(4)$ of the lemma.  As for the remaining properties, we  show $(5)$ and $(7)$. The arguments for $(6)$ and $(8)$ are nearly identical.

For $(5)$, fix $1\leq i\leq m$ and $b\in M^y$, and define 
\[
V'=[\psi_i(x)\geq\eta\wedge\varphi(x,b)\approx_{3\delta}\zeta^\delta_\varphi(b,\cbar_i)].
\]
We want to show $\mu(V')>(1-\gamma)\mu(V_i)$. Note that if $p\in C_i$ then $d(p,p_i)\leq 2\delta$ by $(i)$, and so 
\[
\varphi(p,b)\approx_{2\delta}\varphi(p_i,b)\approx_\delta\zeta^\delta_\varphi(b,\cbar_i).
\]
Thus we have $[\psi_i(x)\geq\eta]\backslash V'\seq V_i\backslash C_i$. Therefore 
\begin{multline*}
\mu([\psi_i(x)\geq\eta]\backslash V')\leq \mu(V_i\backslash C_i)<(\gamma-\tau)\mu(V_i)\leq \\
\frac{\gamma-\tau}{1-\tau}\mu(\psi_i(x)\geq \eta)=\left(1-\frac{1-\gamma}{1-\tau}\right)\mu(\psi_i(x)\geq\eta).
\end{multline*}
So 
\[
\mu(V')>\frac{1-\gamma}{1-\tau}\mu(\psi_i(x)\geq\eta)\geq (1-\gamma)\mu(V_i),
\]
as desired. 

Finally, for $(7)$, fix $(i,j)\in [m]\times [n]$ and set $r=\zeta^\delta_\varphi(q_j,\cbar_i)$. Define 
\[
W'=[\theta_j(y)\geq\eta\wedge\zeta^\delta_\varphi(y,\cbar_i)\approx_{2\delta} r].
\]
We want to  show  $\nu(W')> (1-\gamma)\nu(W_j)$. If $q\in D_j$ then $d(q,q_j)\leq 2\delta$ by $(i)$, and so $\zeta^\delta_\varphi(q,\cbar_i)\approx_{2\delta} r$ by Proposition \ref{prop:lip}. Thus we have $[\theta_j(y)\geq\eta]\backslash W'\seq W_j\backslash D_j$. By choice of $\eta$ and similar steps as above, it follows that $\nu(W')>(1-\gamma)\nu(W_j)$. 
\end{proof}

In order to derive a statement about subsets of $M$ (rather than $S_\varphi(M)$) from the previous lemma, we now assume $M$ is $\omega$-saturated. In this case, any Keisler measure $\mu$ on $S_\varphi(M)$ induces a well-defined finitely  additive probability measure on  the Boolean algebra generated by explicit $\varphi$-zerosets in $M^x$ (see Definition \ref{def:zeroset}). We now restate and prove Theorem \ref{prethm:main-upstairs} from the introduction.

\begin{theorem}\label{thm:main-upstairs}
Suppose $M$ is $\omega$-saturated. Assume $\varphi(x,y)$ is $\delta$-stable, and let $\mu$ and $\nu$ be Keisler measures on $S_\varphi(M)$ and $S_{\varphi^*}(M)$, respectively. Fix some $\epsilon>0$. Then there are $m,n\geq 1$ such that for any $\gamma\in (0,1)$, there are partitions $M^x=A_0\cup A_1\cup\ldots\cup A_m$ and $M^y=B_0\cup B_1\cup\ldots\cup B_n$ satisfying the following properties.
\begin{enumerate}[\hspace{5pt}$\ast$]
\item $A_1,\ldots,A_m$ are explicit $\varphi$-zerosets, and $B_1,\ldots,B_n$ are explicit $\varphi^*$-zerosets. 
\item If $(i,j)\in [m]\times [n]$ then $(A_i,B_j)$ is $(5\delta;\gamma)$-homogeneous for $\mu$ and $\nu$.
\item $\mu(A_0)\leq \epsilon \mu(A_1)$ and $\nu(B_0)\leq\epsilon\nu(B_1)$.
\end{enumerate}
\end{theorem}
\begin{proof}
Apply Lemma \ref{lem:upstairs-hom}. Let $A_i$ be defined by $\psi_i(x)\geq\eta$, and $B_j$ be defined by $\theta_j(y)\geq\eta$. Then $A_1,\ldots,A_m$ are pairwise disjoint explicit $\varphi$-zerosets, and $B_1,\ldots,B_n$ are pairwise disjoint explicit $\varphi^*$-zerosets. Let $A_0=M^x\backslash \bigcup_{i=1}^mA_i$ and $B_0=M^y\backslash \bigcup_{j=1}^nB_j$. Then $\mu(A_0)\leq\epsilon\mu(A_1)$ and $\nu(B_0)\leq\epsilon \nu(B_1)$ by parts $(3)$ and $(4)$ of the lemma. As outlined in Remark \ref{rem:upstairs-hom}, it follows from parts $(5)$-$(8)$ and the triangle inequality that for all $(i,j)\in [m]\times[n]$, the pair $(A_i,B_j)$ is  $(5\delta;\gamma)$-homogeneous for $\mu$ and $\nu$. 
\end{proof}

\begin{remark}\label{rem:ex-vert} 
In the previous theorem, $A_0$ and $B_0$ serve as small ``exceptional" sets, and  are \emph{complements} of zerosets. These sets can be removed at the cost of  weaker control of the error in some of the homogeneous pairs. More specifically, in the context of the theorem, let $A'_1=A_0\cup A_1$. Then for any $1\leq j\leq n$, since $(A_1,B_j)$ is $(5\delta;\gamma)$-homogeneous and $\mu(A_0)\leq\epsilon\mu(A_1)$, it follows that $(A'_1,B_j)$ is $(5\delta;\gamma,\epsilon+\gamma)$-homogeneous. Similarly, if $B'_1=B_0\cup B_1$ then $(A_i,B'_1)$ is $(5\delta;\gamma,\epsilon+\gamma)$-homogeneous for all $1\leq i\leq m$. Moreover, $(A'_1,B'_1)$ is $(5\delta;\epsilon+\gamma)$-homogeneous. On the other hand, $A'_1$ and $B'_1$ are no longer zerosets. In the continuous setting, we expect that is not generally possible to obtain partitions into zerosets with no exceptional sets. However, one could replace $A_0$ with an explicit $\varphi$-zeroset $A'_0$, which contains $A_0$ and  satisfies $\mu(A'_0)\leq (1+\gamma)\mu(A_0)$. After a similar adjustment to $B_0$, this would result in coverings (rather than partitions) of $M^x$ and $M^y$ by finitely many zerosets such that all pairs are homogeneous (with parameters as above), and there is arbitrarily small overlap between the pieces in each partition. 
\end{remark}

\begin{remark}\label{rem:MaPi}
By applying Theorem \ref{thm:main-upstairs} in the setting that $T$ is a classical discrete theory, and $\varphi(x,y)$ is a stable formula, we recover the model-theoretic version of the Malliaris-Shelah stable regularity lemma \cite{MaSh}, proved by Malliaris and Pillay in \cite{MaPi}. However, we have introduced a finer control on the error in the homogeneous pairs by means of the parameter $\gamma$,  \emph{which is allowed to depend on the size of the partition}. This will be reflected in our main result below for stable functions on finite sets (Theorem \ref{thm:main-finite-thm}), where the degree of homogeneity is controlled by an arbitrarily chosen ``decay function". 
\end{remark}

\section{Stable regularity for functions on finite sets}\label{sec:finite}

We now restate Theorem \ref{prethm:main-finite-thm}, which is the main finite result. The proof is given in Section \ref{sec:main-finite-proof} below.

\begin{theorem}\label{thm:main-finite-thm}
Let $V$ and $W$ be finite sets, and suppose $f\colon V\times W\to [0,1]$ is a $(k,\delta)$-stable function. Then for any $\epsilon>0$ and any function $\sigma\colon \N\to (0,1)$,  there are partitions $V=V_0\cup V_1\cup\ldots\cup V_m$ and $W=W_0\cup W_1\cup\ldots\cup W_n$, with $m,n\leq O_{k,\delta,\epsilon,\sigma}(1)$, satisfying the following properties.
\begin{enumerate}[\hspace{5pt}$\ast$]
\item For all $(i,j)\in [m]\times [n]$, the pair $(V_i,W_j)$ is $(5\delta+\epsilon;\sigma(mn))$-homogeneous.

\item $|V_0|\leq\epsilon|V_1|$ and $|W_0|\leq \epsilon|W_1|$.
\end{enumerate}
\end{theorem}

Before proving this result, we make a few remarks. First, using the same calculations as in Remark \ref{rem:ex-vert}, one can remove the exceptional sets $V_0$ and $W_0$ at the cost of weaker error in \emph{some} of the homogeneous pairs. For example, by choosing $\sigma(n)=\frac{1}{2}\epsilon$ in the previous result, we obtain the following simpler (but weaker) version of Theorem \ref{thm:main-finite-thm} without the decay function $\sigma$ or the exceptional sets.

\begin{theorem}\label{thm:main-finite2}
Let $V$ and $W$ be finite sets and suppose $f\colon V\times W\to [0,1]$ is a $(k,\delta)$-stable function. Then for any $\epsilon>0$, there are partitions $V= V_1\cup\ldots\cup V_m$ and $W=W_1\cup\ldots\cup W_n$, with $m,n\leq O_{k,\delta,\epsilon}(1)$, such that $(V_i,W_j)$ is $(5\delta+\epsilon;\epsilon)$-homogeneous for all $(i,j)\in [m]\times [n]$. 
\end{theorem}

The proof of Theorem \ref{thm:main-finite-thm} will also provide strong definability conditions on the sets $V_i$ and $W_j$. In order to give a precise formulation, we first define some terminology. Given a rational number $\alpha=\frac{r}{s}$, with $r,s\in\Z$, $s>0$, and $\gcd(r,s)=1$, define the \emph{complexity} of $\alpha$ to be $\max\{|r|,s\}$.\footnote{This is the standard number-theoretic ``height" function. However, no special properties of this function will be used other than that it defines a map from $\Q$ to $\N$.} The \emph{complexity} of a rational interval is the maximum complexity of its endpoints. 

Now let $f\colon V\times W\to [0,1]$ be a function and fix an integer $N\geq 1$. Then we say that a subset $V'\seq V$ is \emph{$f$-definable of complexity $N$} if
\[
\textstyle V'=\bigcup_{i=1}^m\bigcap_{j=1}^{n_i}\{a\in V:f(a,b_{i,j})\in D_{i,j}\}
\]
for some $b_{i,j}\in W$, some $m,n_1,\ldots,n_m\leq N$, and some closed rational intervals $D_{i,j}$ of complexity at most $N$. We analogously define $f$-definable subsets of $W$. Finally, we define a \emph{min-max $f$-function on $V$ of complexity $N$} to be a function of the form  
\[
\min_{1\leq i\leq m}\max_{1\leq j\leq n_i}f(x,y_{i,j})
\]
 for  some $m,n_1,\ldots,n_m\leq N$. Min-max $f$-functions on $W$ are defined analogously.
 
 \begin{remark}\label{rem:complexity}
 With the above terminology in hand, we can now elaborate on Theorem \ref{thm:main-finite-thm}. In particular, in the conclusion of the theorem, we also have:
\begin{enumerate}[(1)]
\item $V_1,\ldots,V_m, W_1,\ldots,W_n$ are  $f$-definable of complexity $O_{k,\delta,\epsilon,\sigma}(1)$.
\item There are min-max $f$-functions $\zeta_1(y,\xbar)$ on $W$ and $\zeta_2(x,\ybar)$ on $V$ of complexity $O_{k,\delta,\epsilon,\sigma}(1)$, and tuples $\cbar_1,\ldots,\cbar_n\in V^{|\xbar|}$ and $\dbar_1,\ldots,\dbar_m\in W^{|\ybar|}$ satisfying the following properties.
\begin{enumerate}[$\ast$]
\item If $1\leq i\leq m$, then for all $b\in W_s$, 
\[
|\{a\in V_i:f(a,b)\approx_{3\delta+\frac{1}{2}\epsilon}\zeta_1(b,\cbar_i)\}|>(1-\sigma(mn))|V_i|
\]
\item If $1\leq j\leq n$, then for all $a\in V_s$,
\[
|\{b\in W_j:f(a,b)\approx_{3\delta+\frac{1}{2}\epsilon}\zeta_2(a,\dbar_j)\}|>(1-\sigma(mn))|W_j|.
\]

\item For all $(i,j)\in [m]\times [n]$, there are rational numbers $r_{i,j},s_{i,j}\in [0,1]$ of complexity $O_{k,\delta,\epsilon,\sigma}(1)$ such that,  if we define
\begin{align*}
W'_{i,j} &\coloneqq \{b\in W_j:\zeta_1(b,\cbar_i)\approx_{2\delta+\frac{1}{2}\epsilon}r_{i,j}\}\text{ and }\\
V'_{i,j} &\coloneqq \{a\in V_i:\zeta_2(a,\dbar_j)\approx_{2\delta+\frac{1}{2}\epsilon}s_{i,j}\},
\end{align*}
then $|W'_{i,j}|>(1-\sigma(mn))|W_j|$ and $|V'_{i,j}|>(1-\sigma(mn))|V_i|$.
\end{enumerate}
In particular, it follows from the above properties that for all $(i,j)\in [m]\times [n]$, $(V_i,W_j)$ is $(5\delta+\epsilon,\sigma(mn))$-homogeneous, witnessed by $r_{i,j}$, $s_{i,j}$, $W'_{i,j}$, and $V'_{i,j}$. 
\end{enumerate}
See Remark \ref{rem:complexity-proof} for a summary of how the proof yields these extra details.
\end{remark}

The decay function $\sigma$ in Theorem \ref{thm:main-finite-thm} leads to  strong control of the error in the homogeneous pairs. This will be used to match our work to the setting of analytic regularity for functions (see Section \ref{sec:analytic}). Another application of the decay function will appear in Section \ref{sec:equip}, where we modify Theorem \ref{thm:main-finite-thm} so that it yields equipartitions (at the cost of the definability described in Remark \ref{rem:complexity}). 

We now start toward the proof of Theorem \ref{thm:main-finite-thm}. Despite the similarity between this result and Theorem \ref{thm:main-upstairs}, the proof will not be as straightforward as corresponding results in discrete logic. There are essentially three reasons for this. The first is that the zerosets in Theorem \ref{thm:main-upstairs} are not necessarily definable (in the strict sense of continuous logic), and so we need to argue directly with the underlying formulas used to construct these zerosets. The second complication has to do with  the small discrepancy that exists between an abstract ultralimit of normalized counting measures on finite sets when compared to the  pseudofinite normalized average value functional (see Lemma \ref{lem:blur} below). Finally, rather than working directly with homogeneous pairs (which involve an ``if-then" statement with measures), it will be much cleaner to instead transfer the individual underlying components that control homogeneity, as given by Lemma \ref{lem:upstairs-hom}.

\subsection{Approximating the counting measure}\label{sec:transfer}
Let $\Def_{\Q}(\R,<)$ be the set of $D\seq\R$ that are first-order definable over $\Q$ in $(\R,<)$. 

Fix a countable continuous language $\cL$. Let $\cF$ be a countable set of (finitary) $\cL$-formulas, which is closed under the connectives  $\alpha_D$ for $D\in \Def_{\Q}(\R,<)$ (see Definition \ref{def:alphaU}). We define an expanded language $\cL^+$ consisting of $\cL$ together with the following new symbols.
\begin{enumerate}[$(i)$]
\item For each $\psi(\xbar)\in\cF$, add   a new $[0,1]$-valued predicate symbol $\boldsymbol{1}^+_{\psi}(\xbar)$.
\item For each rational $r\in [0,1]$ and  $\theta(x,\ybar)\in\cF$, with $x$ a singleton, add  new $[0,1]$-valued predicate symbols $P_{\theta,r}(\ybar)$ and $Q_{\theta,r}(\ybar)$.
\end{enumerate}
Each new predicate is given a trivial modulus of uniform continuity.

Let $M$ be a finite $\cL$-structure  with the discrete metric. We expand $M$ to an $\cL^+$-structure $M^+$ as follows. For each $\psi(\xbar)\in\cF$, interpret $\boldsymbol{1}^+_\psi(\xbar)$ in $M^+$ as the indicator function of $\psi^M(\xbar)>0$. For each rational $r\in [0,1]$ and $\theta(x,\ybar)\in \cF$, define 
\begin{align*}
P^{M^+}_{\theta,r}\colon M^{\ybar}\to [0,1] &\textnormal{ so that }\bbar\mapsto |\theta^M(x,\bbar)\leq r|/|M^x|,\text{ and }\\
Q^{M^+}_{\theta,r}\colon M^{\ybar}\to [0,1] &\textnormal{ so that }\bbar\mapsto |\theta^M(x,\bbar)\geq r|/|M^x|.
\end{align*}

Now let $(M_s)_{s\in\N}$ be a collection of finite $\cL$-structures with discrete metrics. Fix an ultrafilter $\cU$ on $\N$ and let $M^+=\prod_{\cU}M^+_s$. Let $M=\prod_{\cU}M_s$ be the reduct of $M^+$ to $\cL$. Note that $\cL$ and $\cL^+$ are both countable, and so $M$ and $M^+$ are both $\omega_1$-saturated by Fact \ref{fact:ultrasat}. We identify $M$ with its underlying universe, and use $M^+$ when emphasis on the expanded language is necessary.

Given a sort $x$, let $\mu^s_x$ and $\lf^s_x$ denote the normalized counting measure and average value functional on $(M_s^+)^x$. Let $\lf_x=\lim_{\cU}\lf^s_x$ be the pseudofinite average value functional, and let 
 $\mu_x$ denote the Keisler measure induced by $\lf_x$.  When there is no possibility of confusion, we will omit the subscript $x$ in the previous notation. 
 
 Given an $\cL^+$-formula $\psi(x,\bbar)$ over $M$ and a set $B\seq\R$, define
\[
\mu^*(\psi(x,\bbar)\in B)\coloneqq\lim_{\cU}\mu^s(\psi^{M_s}(x,\bbar^s)\in B),
\]
where $(\bbar^s)_{s\geq 0}$ is a choice of representative for $\bbar$. Note that this is well-defined since the metric on $M$ is discrete. 

\begin{lemma}\label{lem:blur}
Let $\psi(x,\bbar)$ be an $\cL$-formula over $M$, with $\psi(x,\ybar)\in\cF$.
\begin{enumerate}[$(a)$]
\item If $U\in \Def_{\Q}(\R,<)$ is open then $\mu(\psi(x,\bbar)\in U)\leq\mu^*(\psi(x,\bbar)\in U)$.
\item If $C\in \Def_{\Q}(\R,<)$ is closed then $\mu^*(\psi(x,\bbar)\in C)\leq\mu(\psi(x,\bbar)\in C)$.
\end{enumerate} 
\end{lemma}
\begin{proof}
Note that  if $C\seq \R$ is closed then $\mu(\psi(x,\bbar)\in C)=1-\mu(\psi(x,\bbar)\in\R\backslash C)$ (and similarly for $\mu^*$). Moreover, if $U\seq \R$ is open, then $\psi(x,\ybar)\in U$ is logically equivalent to $\alpha_U(\psi(x,\ybar))>0$. So in light of the assumptions on $\cF$, it suffices to just prove part $(a)$, and only consider the case of $\psi(x,\bbar)>0$ (i.e., $U=(0,\infty)$). 

The predicate $\boldsymbol{1}^+_{\psi}(x,\bbar)$ induces a continuous $\{0,1\}$-valued function on $S_x(M^+)$. Let $X\seq S_x(M^+)$ denote the support of this function. Then
\begin{multline*}
\mu(X)=\int_{S_x(M^+)}\boldsymbol{1}^+_{\psi}(x,\bbar)\,~d\mu=\lf(\boldsymbol{1}^+_\psi(x,\bbar))=\lim_{\cU}\lf^s(\boldsymbol{1}^+_{\psi}(x,\bbar^s))\\
=\lim_{\cU}\mu^s(\psi^{M_s}(x,\bbar^s)>0)=\mu^*(\psi(x,\bbar)>0).
\end{multline*}
So to prove the result it suffices to show $[\psi(x,\bbar)>0]\seq X$. 

Fix $p\in[\psi(x,\bbar)>0]$, and set $r=\psi(p,\bbar)$. Suppose $a\in M^x$ is such that $|\psi(a,\bbar)-r|\leq\frac{r}{2}$. Since $\psi(a,\bbar)=\lim_{\cU}\psi(a^s,\bbar^s)$, it follows that for $\cU$-many $s\geq 0$, we have $|\psi(a^s,\bbar^s)-r|<r$. So for $\cU$-many $s\geq 0$, we have $\psi(a^s,\bbar^s)>0$, i.e., $\boldsymbol{1}^+_{\psi}(a^s,\bbar^s)=1$. Therefore  $\boldsymbol{1}^+_{\psi}(a,\bbar)=\lim_{\cU}\boldsymbol{1}^+_{\psi}(a^s,\bbar^s)=1$. Altogether, for any $a\in M^x$ if $|\psi(a,\bbar)-r|\leq\frac{r}{2}$ then $\boldsymbol{1}^+_{\psi}(a,\bbar)=1$. 

Now, since $M^+$ is $\omega$-saturated, it follows (e.g., via \cite[Proposition 7.14]{BBHU}) that for any $q\in S_x(M^+)$, if $|\psi(q,\bbar)-r|\leq\frac{r}{2}$ then $\boldsymbol{1}^+_{\psi}(q,\bbar)=1$. Since $\psi(p,\bbar)=r$, we therefore have $\boldsymbol{1}_{\psi}(p,\bbar)=1$, i.e., $p\in X$, as desired. 
\end{proof}

Note that the previous lemma did not involve the extra predicates introduced above in $(ii)$. These will be used in the next subsection.

\subsection{Proof of Theorem \ref{thm:main-finite-thm}}\label{sec:main-finite-proof}

Suppose the theorem fails for some fixed $k\geq 1$, $\delta,\epsilon>0$, and $\sigma\colon \N\to (0,1)$. Then for all $s\geq 0$, we have a $(k,\delta)$-stable function $f_s\colon V_s\times W_s\to [0,1]$ that admits no partition as in the statement of Theorem \ref{thm:main-finite-thm}, with $m,n\leq s$.

Let $\cL$ be a continuous language with two sorts $V$ and $W$, along with a $[0,1]$-valued binary predicate symbol $f$ on $V\times W$ with trivial modulus of uniform continuity.  For each $s\geq 0$, define an $\cL$-structure $M_s$ such that $V(M_s)=V_s$, $W(M_s)=W_s$, and  $f^{M_s}=f_s$. We equip $M_s$ with the discrete metric. 

Let $\cF$ be a countable set of $\cL$-formulas which contains $f(x,y)$ and is closed under variable substitution as well as the connectives $\max$, $\min$, $|\psi -\theta |$, $\psi\dotminus r$ and $r\dotminus \psi$ for rational $r$, and $\alpha_D$ for $D\in \Def_{\Q}(\R,<)$ (see Definition \ref{def:alphaU}).   Let $M^+_s$ be the expansion of $M_s$ to an $\cL^+$-structure as described in Section \ref{sec:transfer}. Fix a nonprincipal ultrafilter $\cU$ on $\N$, and let $M^+=\prod_{\cU}M^+_s$.  Let $M=\prod_{\cU}M_s$ be the reduct of $M^+$ to $\cL$, and let $\mu^s_x$, $\mu_x$,  and $\mu^*_x$ be as defined in Section \ref{sec:transfer}. We will only use the case where $x$ is a singleton (in either $V$ or $W$), and so we will just write $\mu$, $\mu^s$, and $\mu^*$ (the relevant sort $V$ or $W$ will be clear from context). Set $T=\Th(M)$.

\begin{lemma}\label{lem:finite-1}
$f(x,y)$ is $\delta'$-stable in $T$ for any $\delta'>\delta$.
\end{lemma}
\begin{proof}
Set $\xbar=(x_1,\ldots,x_k)$ and $\ybar=(y_1,\ldots,y_k)$.  Fix $\delta'>\delta$ and define the formula
\[
\textstyle\theta(\xbar,\ybar)\coloneqq \max_{i<j}(\delta'\dotminus |f(x_i,y_j)-f(x_j,y_i)|).
\]
Then for any $s\geq 0$ and $\abar\in V_s^k$, $\bbar\in W_s^k$, we have $\theta(\abar,\bbar)>\delta'-\delta>0$. Therefore $\inf^M_{\xbar,\ybar}\theta(\xbar,\ybar)\geq \delta'-\delta>0$. So $f(x,y)$ is $\delta'$-stable in $T$ (via Fact \ref{fact:ultrasat}).
\end{proof}

We now view $\mu$ as determining local Keisler measures on $S_f(M)$ and on $S_{f^*}(M)$ (specifically, we will work with the pushforward of $\mu$ to these local type spaces, while still using the symbol $\mu$). In the next lemma, we transfer the properties of Lemma \ref{lem:upstairs-hom} to obtain suitable statements in the finite setting. 

\begin{lemma}\label{lem:finite-2}
Let $\psi(x,\bbar)$ be an explicit $f$-formula over $M$, and let $\theta(y,\abar)$ and $\zeta(y,\cbar)$ be explicit $f^*$-formulas over $M$. Fix $\delta',\gamma\in (0,1)$ and $\eta>\lambda>0$. Set $\rho=\eta-\lambda$. For $s\in\N$, define
\[
V_{s,*}=\{a\in V_s:\psi(a,\bbar^s)\geq\lambda\}\mand W_{s,*}=\{b\in W_s:\theta(b,\abar^s)\geq\lambda\}.
\]
\begin{enumerate}[$(a)$]
\item Suppose that for all $b\in M^y$,
\[
\mu\!\left(\psi(x,\bbar)\geq\eta\wedge f(x,b)\approx_{\delta'}\zeta(b,\cbar)\right)~>~(1-\gamma)\mu(\psi(x,\bbar)>0).
\]
Then for $\cU$-many $s\in\N$ we have that for all $b\in W_s$,
\[
|\{a\in V_{s,*}:f(a,b)\approx_{\delta'+\rho}\zeta(b,\cbar^s)\}|>(1-2\gamma)|V_{s,*}|.
\]
\item Suppose that there is some $r\in [0,1]$ such that
\[
\mu(\theta(y,\abar)\geq\eta\wedge\zeta(y,\cbar)\approx_{\delta'}r)~>~(1-\gamma)\mu(\theta(y,\abar)>0).
\]
Then for $\cU$-many $s\in\N$, we have
\[
|\{b\in W_{s,*}:\zeta(b,\cbar^s)\approx_{\delta'+\rho}r\}|>(1-\gamma)|W_{s,*}|.
\]
\end{enumerate}
\end{lemma}
\begin{proof}
Without loss of generality, we may assume $\delta'$, $\eta$, and $\lambda$ are rational.

Part $(a)$. First, note that our assumptions imply $\mu(\psi(x,\bbar)\geq\eta)>0$. So by Lemma \ref{lem:blur}, we have $\mu^*(\psi(x,\bbar)\geq\lambda)>0$. Let $\tau=\frac{1}{2}\mu^*(\psi(x,\bbar)\geq\lambda)$. Then there is some $X_1\in\cU$ such that  $\mu^s(\psi(x,\bbar^s)\geq\lambda)>\tau$ for all $s\in X_1$.

Now define the $\cL$-formula
\[
\xi(x,y,\ybar,\zbar) \coloneqq \max\{\eta\dotminus \psi(x,\ybar),|f(x,y)-\zeta(y,\zbar)|\dotminus \delta'\},
\]
which is in $\cF$. By assumption, for all $b\in M^y$ we have 
\[
\mu(\xi(x,b,\bbar,\cbar)=0)~>~(1-\gamma)\mu(\psi(x,\bbar)>0).
\]
By Lemma \ref{lem:blur}, for all $b\in M^y$ we have
\[
\mu^*(\xi(x,b,\bbar,\cbar)\leq\rho)~>~(1-\gamma)\mu^*(\psi(x,\bbar)\geq\lambda).
\]
So if we define the $\cL^+$-formula 
\[
\pi(y,\ybar,\zbar)\coloneqq (1-\gamma)Q_{\psi,\lambda}(\ybar)\dotminus P_{\xi,\rho}(y,\ybar,\zbar),
\]
then  $\sup^{M^+}_y \pi(y,\bbar,\cbar)=0$. By {\L}o\'{s}'s Theorem, there is some $X_2\in \cU$ such that if $s\in X_2$ then $\sup^{M^+_s}_y\pi(y,\bbar^s,\cbar^s)\leq\gamma\tau$. 
Therefore, if $s\in X_1\cap X_2$ then, for all $b\in W_s$, 
\[
\mu^s(\xi(x,b,\bbar^s,\cbar^s)\leq \rho)~\geq~ (1-\gamma)\mu^s(\psi(x,\bbar^s)\geq\lambda)-\gamma\tau~>~(1-2\gamma)\mu^s(\psi(x,\bbar^s)\geq\lambda), 
\]
and so 
$|\{a\in V_{s,*}:f(a,b)\approx_{\delta'+\rho}\zeta(b,\cbar^s)\}|>(1-2\gamma)|V_{s,*}|$.

Part $(b)$. Define the $\cL$-formula
 \[
 \chi(y,\xbar,\zbar) \coloneqq \max\{\eta\dotminus\theta(y,\xbar),|\zeta(y,\zbar)-r|\dotminus \delta'\}.
 \]
Then $\mu(\chi(y,\abar,\cbar)=0)>(1-\gamma)\mu(\theta(y,\abar)>0)$ by  assumption.  By Lemma \ref{lem:blur}, 
\[
\mu^*(\chi(y,\abar,\cbar)\leq \rho)~\geq~\mu(\chi(y,\abar,\cbar)=0)~>~(1-\gamma)\mu^*(\theta(y,\abar)\geq\lambda).
\]
So for $\cU$-many $s\in\N$, we have
\[
\mu^s(\chi(y,\abar^s,\cbar^s)\leq \rho)~>~(1-\gamma)\mu^s(\theta(y,\abar^s)\geq\lambda),
\]
i.e., $|\{b\in W_{s,*}:\zeta(b,\cbar^s)\approx_{\delta'+\rho}r\}|>(1-\gamma)|W_{s,*}|$. 
\end{proof}

 Choose some $\delta'\in (\delta,\delta+\frac{1}{6}\epsilon)$. By Lemma \ref{lem:finite-1}, $f(x,y)$ is $(k,\delta')$-stable in $T$. We apply Lemma \ref{lem:upstairs-hom}  (with our fixed $\epsilon>0$) to obtain some $m,n\geq 1$, and then choose $\gamma\coloneqq\frac{1}{2}\sigma(mn)$ in the conclusion. For the reader's convenience, we  reiterate the full statement. Set $\zeta_1(y,\xbar)\coloneqq \zeta^{\delta'}_f(y,\xbar)$ and $\zeta_2(x,\ybar)\coloneqq \zeta^{\delta'}_{f^*}(x,\ybar)$.
 Then there are:
\begin{enumerate}[\hspace{5pt}$\ast$]
\item explicit $f$-formulas $\psi_1(x,\bbar_1),\ldots,\psi_m(x,\bbar_m)$ over $M$,
\item explicit $f^*$-formulas $\theta_1(y,\abar_1),\ldots,\theta_n(y,\abar_n)$ over $M$,
\item finite tuples $\cbar_1,\ldots,\cbar_m,\dbar_1,\ldots,\dbar_n$ from $M$,  and
\item some $\eta>0$,
\end{enumerate}
satisfying the following properties.
\begin{enumerate}[$(1)$]
\item $[\psi_1(x,\bbar_1)>0],\ldots,[\psi_m(x,\bbar_m)>0]$ are pairwise disjoint.
\item $[\theta_1(y,\abar_1)>0],\ldots,[\theta_n(y,\abar_n)>0]$ are pairwise disjoint.
\item $\mu(\bigvee_{i=1}^m\psi_i(x,\bbar_i)\geq\eta)~>~1-\epsilon\mu(\psi_1(x,\bbar_1)\geq\eta)$.
\item $\mu(\bigvee_{j=1}^n \theta_j(y,\abar_j)\geq\eta)~>~1-\epsilon\mu(\theta_1(y,\abar_1)\geq\eta)$.
\item If $1\leq i\leq m$ then, for all $b\in M^y$,
\[
\mu\!\left(\psi_i(x,\bbar_i)\geq\eta\wedge f(x,b)\approx_{3\delta'}\zeta_1(b,\cbar_i)\right)~>~(1-\gamma)\mu(\psi_i(x,\bbar_i)>0).
\]
\item If  $1\leq j\leq n$ then, for all $a\in M^x$, 
\[
\mu\!\left(\theta_j(y,\cbar_j)\geq\eta\wedge f(a,y)\approx_{3\delta'}\zeta_2(a,\dbar_j)\right)~>~(1-\gamma)\mu(\theta_j(y,\cbar_j)>0).
\]
\item For all $(i,j)\in [m]\times [n]$, there is some $r_{i,j}\in [0,1]$ such that 
\[
\mu\!\left(\theta_j(y,\cbar_j)\geq\eta\wedge\zeta_1(y,\cbar_i)\approx_{2\delta'} r_{i,j}\right)~>~(1-\gamma)\mu(\theta_j(y)>0).
\]
\item For all $(i,j)\in [m]\times [n]$, there is some $s_{i,j}\in [0,1]$ such that 
\[
\mu\!\left(\psi_i(x,\bbar_i)\geq\eta\wedge\zeta_2(x,\dbar_j)\approx_{2\delta'} s_{i,j}\right)~>~(1-\gamma)\mu(\psi_i(x,\bbar_i)>0).
\]
\end{enumerate}
Note that properties $(3)$-$(8)$  remain true if $\eta$ is replaced by something smaller. So we may assume $\eta\leq 6\delta-6\delta'+\epsilon$ (the right  side is positive by choice of $\delta'$).  Let $\abar=(\abar_1,\ldots,\abar_n)$ and $\bbar=(\bbar_1,\ldots,\bbar_m)$, and define
\[
\psi(x,\bbar)=\min_{1\leq i\leq m}\psi_i(x,\bbar_i)\mand \theta(y,\abar)=\min_{1\leq j\leq n}\theta_j(y,\abar_j).
\]
Given $s\geq 0$, $1\leq i\leq m$, and $1\leq j\leq n$, set $V_{s,i}=\{a\in V_s:\psi_i(a,\bbar^s_i)\geq\frac{1}{2}\eta\}$ and $W_{s,j}=\{b\in W_s:\theta_j(b,\abar^s_j)\geq\frac{1}{2}\eta\}$.

\begin{Claim}\label{cl:finite-3}
The following properties hold for $\cU$-many $s\in\N$. 
\begin{enumerate}[$(a)$]
\item $V_{s,1},\ldots,V_{s,m}$ are pairwise disjoint and $W_{s,1},\ldots,W_{s,n}$ are pairwise disjoint.
\item $|\bigcup_{i=1}^m V_{s,i}|>|V_s|-\epsilon|V_{s,1}|$ and $|\bigcup_{j=1}^n W_{s,j}|>|W_s|-\epsilon|W_{s,1}|$.
\item If $1\leq i\leq m$, then for all $b\in W_s$,
\[
|\{a\in V_{s,i}:f(a,b)\approx_{3\delta+\frac{1}{2}\epsilon}\zeta_1(b,\cbar_i^s)\}|>(1-\sigma(mn))|V_{s,i}|.
\]
\item If $1\leq j\leq n$, then for all $a\in V_s$,
\[
|\{b\in W_{s,j}:f(a,b)\approx_{3\delta+\frac{1}{2}\epsilon}\zeta_2(a,\dbar_j^s)\}|>(1-\sigma(mn))|W_{s,j}|.
\]

\item For all $(i,j)\in [m]\times [n]$, $|\{b\in W_{s,j}:\zeta_1(b,\cbar^s_i)\approx_{2\delta+\frac{1}{2}\epsilon}r_{i,j}\}|>(1-\sigma(mn))|W_{s,j}|$.
\item For all $(i,j)\in [m]\times[n]$, $|\{a\in V_{s,i}:\zeta_2(a,\dbar^s_j)\approx_{2\delta+\frac{1}{2}\epsilon}s_{i,j}\}|>(1-\sigma(mn))|V_{s,i}|$.
\end{enumerate}
\end{Claim}
\begin{proof}
Since the metric on $M$ is discrete, part $(a)$ follows easily  from {\L}o\'{s}'s Theorem and properties $(1)$ and $(2)$ above. For part $(b)$, note that by $(3)$ and  Lemma \ref{lem:blur}, we have $\mu^*(\bigvee_{i=1}^m\psi_i(x,\bbar_i)\geq\frac{1}{2}\eta)>1-\epsilon\mu^*(\psi_1(x,\bbar_i)\geq\frac{1}{2}\eta)$, and so $|\bigcup_{i=1}^m V_{s,i}|>|V_s|-\epsilon|V_{s,1}|$ holds for $\cU$-many $s\in\N$. We similarly get $|\bigcup_{j=1}^n W_{s,j}|>|W_s|-\epsilon|W_{s,1}|$ for $\cU$-many $s\in\N$ from $(4)$ and Lemma \ref{lem:blur}. 
For parts $(c)$ through $(f)$, apply Lemma \ref{lem:finite-2} to  $(5)$ through $(8)$, while choosing $\lambda=\frac{1}{2}\eta$. 
\end{proof}

Since $\cU$ is nonprincipal, we may choose $s\geq m,n$ satisfying the properties in the previous claim. Set $V_{s,0}=V_s\backslash \bigcup_{i=1}^m V_{s,i}$ and $W_{s,0}=W_s\backslash\bigcup_{j=1}^n W_{s,j}$. Then $|V_{s,0}|< \epsilon |V_{s,1}|$ and $|W_{s,0}|<\epsilon|W_{s,1}|$ by part $(b)$ of Claim \ref{cl:finite-3}. By parts $(c)$, $(d)$, $(e)$, and $(f)$ of Claim \ref{cl:finite-3}, and the triangle inequality, it follows that for all $(i,j)\in [m]\times[n]$, $(V_{s,i},W_{s,j})$ is $(5\delta+\epsilon;\sigma(mn))$-homogeneous. Altogether, this contradicts the choice of $f_s\colon V_s\times W_s\to [0,1]$, and we have finished the proof of Theorem \ref{thm:main-finite-thm}.

\begin{remark}\label{rem:complexity-proof}
In order to obtain the extra definability conditions described in Remark \ref{rem:complexity}, one only needs to further assume that $f_s$ admits no such partition in which the complexity of the ingredients is bounded by $s$. Then at the end of the proof, choose $s$ to be larger than the complexity of the objects constructed. This requires one to also assume that various parameters are rational, in particular, $\eta$, $r_{i,j}$, $s_{i,j}$, and the endpoints of the intervals involved in each $\psi_i$ and $\theta_j$.  For $\eta$, this is easy, and for $\psi_i$ and $\theta_j$ use Remark \ref{rem:basic-open}. Finally, replace  $r_{i,j}$ and $s_{i,j}$ by a sufficiently close rational number, and use $\epsilon$ to absorb the difference. 
\end{remark}

\section{Further results}

\subsection{Equipartitions}\label{sec:equip}

A common tension between regularity lemmas proved using finitary methods versus those proved using model theoretic methods is that the model theoretic proofs typically do not provide equipartitions without further work.  In the case of regularity for \emph{arbitrary} graphs or functions, a partition can be turned into an equipartition using a number of standard methods. However, for stable regularity, which involves homogeneity and \emph{and no irregular pairs}, more care is required to build an equipartition. In this section, we will demonstrate how the decay function in Theorem \ref{thm:main-finite-thm} makes this relatively easy.\footnote{This was also observed by Terry in the setting of stable graph regularity.} On the other hand, we note that the process of turning a partition into an equipartition usually results in a loss of  ``definability" of the pieces (as in Remark \ref{rem:complexity}).

\begin{theorem}\label{thm:equip} 
Let $V$ and $W$ be finite sets, and suppose $f\colon V\times W\to [0,1]$ is a $(k,\delta)$-stable function. Then for any $\epsilon>0$ and any function $\sigma\colon \N\to (0,1)$, there are partitions $V=V_0\cup V_1\cup\ldots\cup V_m$ and $W=W_0\cup W_1\cup\ldots\cup W_n$, with $m,n\leq O_{k,\delta,\epsilon,\sigma}(1)$, satisfying the following properties.
\begin{enumerate}[$(i)$]
\item For all $(i,j)\in [m]\times [n]$, the pair $(V_i,W_j)$ is $(5\delta+\epsilon;\sigma(mn))$-homogeneous.
\item $|V_i|=|V_j|$ for all $1\leq i,j\leq m$; and $|W_i|=|W_j|$ for all $1\leq i,j\leq n$.
\item $|V_0|\leq\epsilon|V|$ and $|W_0|\leq \epsilon|W|$.
\end{enumerate}
\end{theorem}
\begin{proof}
Without loss of generality, assume $\sigma$ is decreasing. 
Let $\tau\colon \N\to (0,1)$ be defined by $\tau(n)=\frac{\epsilon}{2n}\sigma(4n^2\lceil\epsilon\inv\rceil^2)$, and let $N$ be the  bound $O_{k,\delta,\epsilon/2,\tau}(1)$ from Theorem \ref{thm:main-finite-thm}.  Now let $f\colon V\times W\to [0,1]$ be $(k,\delta)$-stable. 
By Theorem \ref{thm:main-finite-thm}, there are partitions $V=V'_0\cup V'_1\cup\ldots\cup V'_{m'}$ and $W=W'_0\cup W'_1\cup\ldots\cup W'_{n'}$, with $m',n'\leq N$, such that $|V'_0|\leq\epsilon|V'_1|/2$, $|W'_0|\leq\epsilon|W'_1|/2$, and $(V'_i,W'_j)$ is $(5\delta+\epsilon;\tau(m'n'))$-homogeneous for all $(i,j)\in [m']\times[n']$. Set $N^*=2N^2\epsilon\inv=O_{k,\delta,\epsilon,\sigma}(1)$.

For each $1\leq i\leq m'$, partition $V'_i=V'_{i,1}\cup\ldots\cup V'_{i,t_i}\cup X_i$ so that $|V'_{i,p}|=\lceil \frac{\epsilon}{2m'}|V|\rceil$ for all $1\leq p\leq t_i$, and $|X_i|\leq \frac{\epsilon}{2m'}|V|$ (note that we allow $t_i=0$). Similarly, for each $1\leq j\leq n'$, partition $W'_j=W'_{j,1}\cup\ldots\cup W'_{j,u_j}\cup Y_j$ so that $|W'_{i,q}|=\lceil\frac{\epsilon}{2n'}|W|\rceil$ and $|Y_j|\leq \frac{\epsilon}{2n'}|W|$. Let $V_1,\ldots,V_m$ enumerate $\{V'_{i,p}:1\leq i\leq m',~1\leq p\leq t_i\}$; and let $W_1,\ldots,W_n$ enumerate $\{W'_{j,q}:1\leq j\leq n',~1\leq q\leq u_j\}$. Then $(ii)$ holds by construction. For each $1\leq i\leq m'$, we have $t_i\frac{\epsilon}{2m'}|V|\leq |V'_i|\leq |V|$, and so $t_i\leq 2m'\epsilon\inv$. Thus $m\leq 2(m')^2\epsilon\inv\leq N^*$. Similarly, $n\leq 2(n')^2\epsilon\inv \leq N^*$.

To show $(i)$, we fix $V'_{i,p}$ and $W'_{j,q}$, and show that $(V'_{i,p},W'_{j,q})$ is $(5\delta+\epsilon;\sigma(mn))$-homogeneous. By construction, there are  $r\in [0,1]$ and $W'\seq W'_j$ such that $|W'|\geq (1-\tau(m'n'))|W'_j|$ and, for all $b\in W'$,
\begin{equation*}
|\{a\in V'_i:f(a,b)\approx_{5\delta+\epsilon}r\}|\geq (1-\tau(m'n'))|V'_i|.\tag{$\dagger$}
\end{equation*}
Let $W''=W'\cap W'_{j,q}$. Then 
\[
|W''|\geq |W'_{j,q}|-\tau(m'n')|W'_j|\geq|W'_{j,q}|-\tau(m'n')|W|
\textstyle \geq|W'_{j,q}|-\tau(m'n')\frac{2n'}{\epsilon}|W'_{i,q}|.
\]
Recall that $mn\leq 4(m'n')^2\epsilon^{\nv 2}$, and so $\tau(m'n')\frac{2n'}{\epsilon}\leq \sigma(mn)$ by choice of $\tau$. So $|W''|\geq (1-\sigma(mn))|W'_{i,q}|$. Moreover, if $b\in W''$ then by $(\dagger)$,
\[
|\{a\in V'_{i,p}:f(a,b)\approx_{5\delta+\epsilon}r\}|\geq |V'_{i,p}|-\tau(m'n')|V'_i|\geq (1-\sigma(mn))|V'_{i,p}|,
\]
where the final inequality follows by similar calculations. By a symmetric argument, we obtain the desired homogeneity for $(V'_{i,p},W'_{j,q})$.

Finally, set $V_0=V'_0\cup\bigcup_{i=1}^{m'}X_i$ and $W_0=W'_0\cup\bigcup_{j=1}^{n'}Y_j$. Note that we now have partitions $V=V_0\cup V_1\cup\ldots\cup V_m$ and $W=W_0\cup W_1\cup\ldots\cup W_n$ satisfying $(i)$ and $(ii)$. So it remains to prove the bounds in $(iii)$. 
For this, we have
\[
\textstyle |V_0|= |V'_0|+\sum_{i=1}^{m'}|X_i|\leq \frac{\epsilon}{2}|V'_1|+m'{\left(\frac{\epsilon}{2m'}|V|\right)}\leq\epsilon|V|.
\]
By a similar argument, we get $|W_0|\leq\epsilon|W|$. 
\end{proof}

\begin{remark}
As in Theorem \ref{thm:main-finite-thm}, the sets $V_0$ and $W_0$ are exceptional sets of vertices, which are used to ensure strong homogeneity with a decay function (c.f. Theorem \ref{thm:main-finite2}) \emph{and} to achieve perfectly balanced equipartitions. The use of an exceptional set to achieve the latter feature is typical in general regularity as well (see, e.g., \cite[Theorem 1.7]{KomSim}). A standard alternate approach is to evenly distribute the exceptional set among the remaining pieces of the partition, which yields a new partition satisfying $||V_i|-|V_j||\leq 1$ for all $i,j$ (see, e.g., \cite[Theorem 1.8]{KomSim}). In our situation, one could  do this to remove $V_0$ and $W_0$, but it would again result in homogeneity controlled only by $\epsilon$ rather than $\sigma$ (as in Theorem \ref{thm:main-finite2}). 
\end{remark}

\subsection{The case of graphs}

Continuing with  Remarks \ref{rem:stable-graphs} and \ref{rem:MaPi}, we note that our previous results yield  a qualitative version of stable graph regularity \cite{MaSh}  with homogeneity controlled by a decay function. Our formulation is in terms of bipartitioned graph relations, which differs  from \cite{MaSh} (see Section \ref{sec:bipartite} for further discussion).

\begin{corollary}\label{cor:graph-version}
Let $V$ and $W$ be finite sets and suppose $E\seq V\times W$ is  $k$-stable. Then for any $\epsilon>0$ and any function $\sigma\colon\N\to (0,1)$, there are partitions $V=V_0\cup V_1\cup\ldots\cup V_m$ and $W=W_0\cup W_1\cup\ldots \cup W_n$,  with $m,n\leq O_{k,\epsilon,\sigma}(1)$, such that for all  $(i,j)\in [m]\times [n]$, 
\[
\text{$|E\cap (V_i\times W_j)|\leq \sigma(mn)|V_i||W_j|$ or $|E\cap (V_i\times W_j)|\geq (1-\sigma(mn))|V_i||W_j|$.}
\] 
Moreover, one of the follow cases holds.
\begin{enumerate}[$(i)$]
\item $|V_0|\leq\epsilon|V_1|$, $|W_0|\leq \epsilon|W_1|$, and $V_1,\ldots,V_m, W_1,\ldots,W_n$ are  $E$-definable of complexity $O_{k,\epsilon,\sigma}(1)$.
\item$|V_0|\leq\epsilon|V|$, $|W_0|\leq\epsilon|W|$, $|V_i|=|V_j|$ for all $1\leq i,j\leq m$, and $|W_i|=|W_j|$ for all $1\leq i,j\leq n$.
\end{enumerate}
\end{corollary}
\begin{proof}
Let $f=\boldsymbol{1}_E$. As explained in Remark \ref{rem:stable-graphs}, $f$ is $(O_k(1),\delta)$-stable for any $\delta>0$. So we can apply Theorem \ref{thm:main-finite-thm} with $\delta<\frac{1}{10}-\epsilon$ (without loss of generality, assume $\epsilon<\frac{1}{10}$). It follows  that in a $(5\delta+\epsilon;\sigma(mn))$-homogeneous pair $(V_i,W_j)$, we may take the uniform value $r_{i,j}$ to be either $0$ or $1$. This yields the main claim with case $(i)$ (via Remark \ref{rem:complexity}). To instead obtain case $(ii)$, replace Theorem \ref{thm:main-finite-thm} with  Theorem \ref{thm:equip} in the previous argument.
\end{proof}

A quantitative proof of the previous result for graphs (stated with case $(ii)$) was recently given by Terry and Wolf \cite[Theorem 4.7]{TeWo}.

\subsection{Function decomposition}\label{sec:analytic}
In this section, we use Theorem \ref{thm:main-finite-thm} to derive a stable analogue of the analytic form of Szemer\'{e}di's regularity lemma for functions. We follow the formalism described by Tao in \cite{TaoSR}. Since the analytic setting involves decompositions of functions into \emph{sums} of various components, we will need to develop terminology for functions that are not necessarily $[0,1]$-valued. 

Given a (nonempty) finite set $X$ and functions $f,g\colon X\to \R$,  define the \emph{normalized} inner product 
\[
\textstyle\langle f,g\rangle=\frac{1}{|X|}\sum_{x\in X}f(x)g(x).
\]
The \emph{normalized} $\ell^2$-norm of $f$ is  $\|f\|_2=\sqrt{\langle f,f\rangle}$. Note that  $\|f\|_1=\langle |f|,\boldsymbol{1}\rangle$, and recall also that $\|f\|_\infty=\max_{x\in X}|f(x)|$. So the inequalities $\|f\|_1\leq \|f\|_2\leq\|f\|_\infty$ hold for any $f$ (the first inequality follows from Cauchy-Schwarz).

\begin{definition}
Let $V$ and $W$ be finite sets, and fix  $f\colon V\times W\to [\nv 1,1]$.  
\begin{enumerate}[$(1)$]
 \item  $f$ is  \textbf{$(m,n)$-structured} it is of the form $\sum_{i,j}r_{i,j}\boldsymbol{1}_{V_i\times W_j}$ for some partitions $V=V_0\cup V_1\cup\ldots\cup V_m$ and $W=W_0\cup W_1\cup\ldots\cup W_n$, with   $r_{i,j}\in [\nv 1,1]$. 
\item  $f$ is \textbf{$\epsilon$-pseudorandom} if $|\langle f,\boldsymbol{1}_{A\times B}\rangle|\leq\epsilon$ for all $A\seq V$ and $B\seq W$.  
 \end{enumerate}
 \end{definition}
 
 The previous definitions are adapted from \cite{TaoSR}, although we remark that the  ``structured" notion there applies to $\R$-valued functions, and thus includes an additional parameter $K$ bounding $|r_{i,j}|$ (which we do not require). The notion of pseudorandomness is a special case of the definition from   \cite{TaoSR}, which involves  a  more general schematic. We have specialized to the  ``product structure" setting described in \cite[Example 2.3]{TaoSR}. We can now state the analytic regularity lemma for functions (quoting \cite[Lemma 1.1]{Green-Tao}, with some clarification to follow).

 \begin{lemma}[Szemer\'{e}di's Regularity Lemma, analytic form]\label{lem:SRLa}
 Let $V$ and $W$ be finite sets, and let $f\colon V\times W\to [0,1]$ be a function. Then for any $\epsilon>0$ and $\sigma\colon \N\to (0,1)$, there are $m,n\leq O_{\epsilon,\sigma}(1)$ and a decomposition $f=f_{\str}+f_{\psd}+f_{\err}$ such that $f_{\str}$ is $(m,n)$-structured, $f_{\psd}$ is $\sigma(mn)$-pseudorandom, and $\|f_{\err}\|_2\leq \epsilon$. Moreover, $f_{\psd}$ is $[\nv 1,1]$-valued, while $f_{\str}$ and $f_{\str}+f_{\err}$ are $[0,1]$-valued. 
 \end{lemma}
 
 In \cite{Green-Tao}, pseudorandomness is formulated using the  ``box norm", and the equivalence with the definition above is a fundamental result in graph theory (see, e.g.,  \cite[Theorem 2.4]{GowQRG} for a precise statement in the context of functions). We also note that the previous lemma is stated in a ``bipartite form", which differs from \cite{Green-Tao} (see Section \ref{sec:bipartite} for further discussion).

As a segue to stability, let us  discuss how  classical graph regularity  can be used to obtain a prototype of Lemma \ref{lem:SRLa} for the indicator function of a bipartite graph. Fix  some $E\seq V\times W$ and $\epsilon>0$. Given (nonempty) $A\seq V$ and $B\seq W$, set $\alpha_{A,B}\coloneqq |E\cap (A\times B)|/|A\times B|$ (the density of $E$ on $A\times B$). Then Szemer\'{e}di's regularity lemma provides partitions $V=V_1\cup\ldots\cup V_m$ and $W=W_1\cup\ldots\cup W_n$ (with $m,n\leq O_\epsilon(1)$), and a set $\Sigma$ of pairs $(i,j)$, such that if $Z=\bigcup_{(i,j)\in\Sigma}V_i\times W_j$ then $|Z|\leq\epsilon|V\times W|$, and if $(i,j)\not\in\Sigma$ then  $(V_i,W_j)$ is $\epsilon$-regular. So for $(i,j)\not\in\Sigma$, $\alpha_{i,j}\coloneqq \alpha_{V_i,W_j}\approx_\epsilon \alpha_{A,B}$ for any $A\seq V_i$ and $B\seq W_j$ with $|A|\geq\epsilon|V_i|$ and $|B|\geq\epsilon|W_j|$. An easy calculation then shows that $\boldsymbol{1}_{E\cap (V_i\times W_j)}-\alpha_{i,j}$ is $2\epsilon$-pseudorandom as a function on $V_i\times W_j$. Thus if we set $f_{\str}=\sum_{(i,j)\not\in\Sigma}\alpha_{i,j}\boldsymbol{1}_{V_i\times W_j}$, $f_{\err}=\boldsymbol{1}_{E\cap Z}$, and $f_{\psd}=\boldsymbol{1}_E-f_{\str}-f_{\err}$, then $\boldsymbol{1}_E=f_{\str}+f_{\psd}+f_{\err}$, and we have that $f_{\str}$ is $(m,n)$-structured, $f_{\psd}$ is $2\epsilon$-pseudorandom, and $\|f_{\err}\|_2\leq \sqrt{\epsilon}$. Also,  $f_{\str}$ and $f_{\err}$ are $[0,1]$-valued, while $f_{\psd}$ is $[\nv 1,1]$-valued. 

In the setting of stable \emph{graphs}, the previous situation is qualitatively strengthened in two ways. Firstly, the set $\Sigma$ of irregular pairs does not appear, which removes the $f_{\err}$ term. Secondly, $\epsilon$-regularity is replaced by $\epsilon$-homogeneity, which is to say that each density $\alpha_{i,j}$ is within $\epsilon$ of some $\hat{\alpha}_{i,j}\in\{0,1\}$.\footnote{This is essentially equivalent to a suitable graph-theoretic analogue of functional homogeneity (as in Definition \ref{def:hom}), up to uniform change in $\epsilon$.}  
 Thus $(V_i,W_j;E)$ is almost complete or empty, and $\hat{\alpha}_{i,j}$ is the ``generic" value of $\boldsymbol{1}_E$ on $V_i\times W_j$. Consequently, if one \emph{redefines} $f_{\str}$ above so that $\alpha_{i,j}$ is replaced by $\hat{\alpha}_{i,j}$, then $f_{\str}$ is still $(m,n)$-structured (in fact, it is the indicator function of the union of all $V_i\times W_j$ on which $E$ is almost complete), and $f_{\psd}=f-f_{\str}$ is a $\{\nv 1,0,1\}$-valued function whose support has size at most $\epsilon|V\times W|$. This motivates the following remark.

\begin{remark}\label{rem:supp}
Suppose $f\colon V\times W\to [\nv 1,1]$ is an arbitrary function such that $|{\supp}(f)|\leq\epsilon|V\times W|$. Then $\|f\|_1\leq\epsilon$, which implies that  \emph{for any other function} $g\colon V\times W\to [\nv 1,1]$, we have $|\langle f,g\rangle|\leq\|f\|_1\leq\epsilon$. So $f$ is automatically $\epsilon$-pseudorandom. 
In other words, having small support can be viewed as a very strong form of pseudorandomness.
\end{remark}

In the setting of stable functions, we will see the same picture emerge, in the sense that $f_{\str}$ will be determined by the ``generic value" of $f$ on each homogeneous pair, and $f_{\psd}$ will have small support. At this level, one could obtain from Theorem \ref{thm:main-finite2} a decomposition for stable functions that is very much like the one described above for stable graphs (see Remark \ref{rem:ferr}).
However, note that the above discussion of stable graphs does not include the stronger control on pseudorandomness using a decay function, as in Lemma \ref{lem:SRLa}. In order to re-introduce this aspect, we will need to deal with the issues of the exceptional sets $V_0$ and $W_0$, and the appearance of $\delta$ and $\epsilon$ in the homogeneous pairs. To handle the latter issue, we will relax the  ``structured" component $f_{\str}$ by allowing for some uniformly bounded fluctuation. 

\begin{definition}\label{def:approxstr}
A function $f\colon V\times W\to [0,1]$ is \textbf{$(\delta;m,n)$-structured} if there is some $(m,n)$-structured function $g\colon V\times W\to [0,1]$ such that $\|f-g\|_\infty\leq\delta$. 
\end{definition}

As for the exceptional sets, we will deal with those in the same way as one deals with irregular pairs, which is to put them in an error term. This will yield an error term supported on a set of the form $Z'=(V_0\times W)\cup (V\times W_0)$, where $|V_0|\leq\epsilon|V|$ and $|W_0|\leq\epsilon|W|$. Note that $|Z'|\leq 2\epsilon|V\times W|$. So $Z'$ is comparable to the error set $Z$ above in terms of size, and we similarly get bounded $\ell^2$-norm for any function supported on $Z'$. On the other hand, it is important to note that $Z$ and $Z'$ are \emph{qualitatively} different in a way undetected by norms. Indeed, $Z'$ is built from essentially unary ingredients, while $Z$ is necesssarily binary. More precisely, if a function $f\colon V\times W\to [0,1]$ is decomposed as $g+h$, where $h$ is supported on $Z'$ (and $g$ has some desirable properties), then after removing a small amount of  $V$ and $W$ individually, one can assume $f=g$. On the other hand, if $h$ were supported on $Z$, then one would need to remove some possibly complicated subset of $V\times W$, which represents a more drastic change to the nature of $f$.

To close this discussion, we will say that a function $f\colon V\times W\to [\nv 1,1]$ has \textbf{$\epsilon$-structured support} if its support is contained in a set of the form $Z'$ above. Note that if $f$ has $\epsilon$-structured support then $|{\supp}(f)|\leq 2\epsilon|V\times W|$, and so $\|f\|_2\leq \sqrt{2\epsilon}$.

\begin{theorem}\label{thm:function-version}
Let $V$ and $W$ be finite sets, and suppose $f\colon V\times W\to [0,1]$ is a $(k,\delta)$-stable function. Then for any $\epsilon>0$ and any function $\sigma\colon \N\to (0,1)$, there are $m,n\leq O_{k,\delta,\epsilon,\sigma}(1)$ and a decomposition $f=f_{\str}+f_{\psd}+f_{\err}$ such that:
\begin{enumerate}[$(i)$]
\item $f_{\str}$ is $(5\delta+\epsilon;m,n)$-structured, 
\item $|{\supp}(f_{\psd})|\leq \sigma(mn)|V\times W|$, and 
\item $f_{\err}$ has $\epsilon$-structured support.
\end{enumerate}
 Moreover, $f_{\psd}$ is $[\nv 1,1]$-valued, $f_{\str}$ and $f_{\err}$ are $[0,1]$-valued, and $f_{\err}\in\{ f_{\str},|f_{\psd}|\}^\perp$. 
\end{theorem}
\begin{proof}
Apply Theorem \ref{thm:main-finite-thm} to obtain partitions $V=V_0\cup V_1\cup\ldots\cup V_m$ and $W=W_0\cup W_1\cup\ldots \cup W_n$, with $m,n\leq O_{k,\delta,\epsilon,\sigma}(1)$, so that $|V_0|\leq \epsilon|V_1|$, $|W_0|\leq \epsilon|W_1|$, and for all $(i,j)\in [m]\times [n]$, $(V_i,W_j)$ is $(5\delta+\epsilon;\frac{1}{2}\sigma(mn))$-homogeneous. For each $(i,j)\in [m]\times [n]$ choose $r_{i,j}\in [0,1]$ and $W_{i,j}\seq W_j$ such that $|W_{i,j}|\geq (1-\frac{1}{2}\sigma(mn))|W_j|$ and, for all $b\in W_{i,j}$, $|\{a\in V_i:f(a,b)\approx_{5\delta+\epsilon} r_{i,j}\}|\geq (1-\frac{1}{2}\sigma(mn))|V_i|$. 

Set $Z=(V_0\times W)\cup (V\times W_0)$. Let $f_{\str}\colon V\times W\to [0,1]$ be defined by
\[
f_{\str}(a,b)=\begin{cases}
f(a,b) & \text{if $(a,b)\in V_i\times W_j$ and $f(a,b)\approx_{5\delta+\epsilon}r_{i,j}$,}\\
r_{i,j} & \text{if $(a,b)\in V_i\times W_j$ and $f(a,b)\not\approx_{5\delta+\epsilon}r_{i,j}$, and}\\
0 & \text{if $(a,b)\in Z$.}
\end{cases}
\]
Define $f_{\err}\coloneqq f\boldsymbol{1}_Z$ and $f_{\psd}\coloneqq f-f_{\str}-f_{\err}$. Then other than $(ii)$,  the claims in the theorem follow easily by construction. So we show $(ii)$.

Set $\gamma=\frac{1}{2}\sigma(mn)$. Given $(i,j)\in [m]\times [n]$ and $b\in W$, partition $V_i=X^b_{i,j}\cup Y^b_{i,j}$ so that $a\in X^b_{i,j}$ if and only if $f(a,b)\approx_{5\delta+\epsilon}r_{i,j}$. So if $b\in W_{i,j}$ then $|Y^b_{i,j}|\leq \gamma|V_i|$ and $f_{\psd}(a,b)=0$ for all $a\in X^b_{i,j}$.  It follows that
\[
\textstyle \supp(f_{\psd})\seq \bigcup_{(i,j)\in[m]\times[n]}\big((V_i\times (W_j\backslash W_{i,j}))\cup\bigcup_{b\in W_{i,j}}(Y^b_{i,j}\times\{b\})\big).
\]
So $|{\supp}(f_{\psd})|\leq \sum_{i,j}(\gamma|V_i||W_j|+\gamma|V_i||W_{i,j}|)\leq 2\gamma|V\times W|=\sigma(mn)|V\times W|$.
\end{proof}

\begin{remark}
The previous proof also implies that the underlying partition of $f_{\str}$ involves $f$-definable pieces of bounded complexity. We could instead use Theorem \ref{thm:equip} and obtain an equipartition. (Note that we only need $|V_0|\leq\epsilon|V|$ and $|W_0|\leq\epsilon|W|$ to know that $f_{\err}$ has $\epsilon$-structured support.)
\end{remark}

\begin{remark}\label{rem:discrete}
Call a function $f\colon V\times W\to [0,1]$ \emph{$k$-stable} if it is $(k,\delta)$-stable for some $\delta>0$ such that  for all  $x,y\in V\times W$, if $f(x)\neq f(y)$ then $|f(x)-f(y)|>10\delta$. For example this includes the setting of $k$-stable bipartite graphs (after changing $k$ as discussed in Remark \ref{rem:stable-graphs}). Suppose $f\colon V\times W\to [0,1]$ is $k$-stable, and fix $V'\seq V$ and $W'\seq W$ such that $(V',W')$ is $(5\delta+\epsilon;\epsilon)$-homogeneous for sufficiently small $\epsilon$. Then it follows that there is some $r\in \text{Im}(f)$ such that for all but at most $\epsilon|W'|$-many $b\in W'$, for all but at most $\epsilon|V'|$-many $a\in V'$, $f(a,b)=r$, \emph{and dually}, for all but at most $\epsilon|V'|$-many $a\in V'$, for all but at most $\epsilon|W'|$-many $b\in W'$, $f(a,b)=r$. Altogether, one obtains a version of Theorem \ref{thm:function-version} for $k$-stable functions in which $f_{\str}$ is  $(m,n)$-structured. 
\end{remark}

\begin{remark}\label{rem:ferr}
One can obtain a decomposition of $(k,\delta)$-stable functions involving \emph{no error term}, at the cost of the decay function. Indeed, in Theorem \ref{thm:function-version}, if we replace $f_{\psd}$ with $f_{\psd}+f_{\err}$, then we have $|{\supp}(f_{\psd})|\leq (\sigma(mn)+2\epsilon)|V\times W|$. Note that we can also write $f_{\str}=f'_{\str}+h$ where $f'_{\str}$ is $(m,n)$-structured and $\|h\|_\infty\leq 5\delta+\epsilon$. So if one is willing to allow $f_{\psd}$ to involve $\delta$, then $f_{\str}$ can be made perfectly structured by  replacing $f_{\psd}$ with $f_{\psd}+h$. In this case, $f_{\psd}$ no longer has small support, but one still has  a bound on $\|f_{\psd}\|_1$ in terms of $\delta$, $\epsilon$, and $\sigma(mn)$.  Therefore $f_{\psd}$ is still pseudorandom in a strong qualitative sense (see Remark \ref{rem:supp}). 
\end{remark}

\subsection{Bipartite versus symmetric}\label{sec:bipartite}

In this section, we clarify some subtleties regarding our bipartite viewpoint on graphs and functions. First, recall that any graph $(V;E)$ can be ``coded" as bipartite graph $(V,V;E)$ (sometimes called the \emph{bipartite double cover} of $(V;E)$). Similarly, a $[0,1]$-valued binary function $f$ on a set $V$ can be viewed as a bipartitioned function $f\colon V\times V\to [0,1]$.  From this perspective, the bipartite setting is more flexible since it allows one to distinguish between the two sets. This is also a very natural setting in which to apply model theoretic tools. On the other hand, there is one  issue with the bipartite approach, which is usually not discussed in the model theoretic literature on regularity. In particular,  given a graph $(V;E)$ if one applies a bipartite regularity lemma  to $(V,V;E)$ then this results in two potentially different partitions of $V$.

One way to remedy this issue is to again exploit the decay function. In particular, suppose we have partitions $ V_1\cup\ldots\cup V_m$ and $W_1\cup \ldots\cup W_n$ of the same finite set $V$, in which each pair $(V_i,W_j)$ satisfies a desired homogeneity property with respect to some $f\colon V\times V\to [0,1]$ (we ignore the issue of exceptional sets for the moment). Then we have a common refinement, which partitions $V$ into at  most $mn$ sets. Thus, given a target decay function $\sigma\colon \N\to (0,1)$, one can define a modified function $\tau$ (similar to the proof of Theorem \ref{thm:equip}) so that if  the initial partitions are homogeneous with respect to $\tau$, then any sufficiently large piece of the common refinement maintains homogeneity using $\sigma$, while the remaining small pieces can be put into an exceptional set. Moreover, if we also start out with two exceptional sets $V_0$ and $W_0$, then these can be added to the larger exceptional set, along with any $W_j$ that intersects $V_0$ in a large set (and vice versa).

As a final remark, we note that when $f\colon V\times V\to [0,1]$ is a \emph{symmetric} function, one can also address this issue ``upstairs" at the model-theoretic stage. Indeed, given a continuous structure $M$ and a symmetric $[0,1]$-valued formula $\varphi(x,y)$, with $x$ and $y$ variables of the same sort, the type spaces $S_\varphi(M)$ and $S_{\varphi^*}(M)$ can be naturally identified. Therefore, in the setting of Lemma \ref{lem:upstairs-hom}, if the measures $\mu$ and $\nu$ are the same (for example, if they  both arise from the pseudofinite average value functional), then one can prove the lemma using \emph{the same} $\varphi$-formulas on both sides. Carrying this through the rest of the steps, we obtain a single partition of $V$.

\begin{appendices}

\appendix

\renewcommand*{\thesection}{}
\renewcommand*{\thesubsection}{\Alph{section}.\arabic{subsection}}
\renewcommand{\thetheorem}{\Alph{section}.\arabic{theorem}}
\section{}

We  discuss a natural variation of $(k,\delta)$-stability (as defined in Definition \ref{def:stable}), which is closer to the definition of $k$-stability for bipartite graphs. 

\begin{definition}\label{def:stable-alt}
A function $f\colon V\times W\to [0,1]$ is \textbf{${}^*\!(k,\delta)$-stable} if there do not exist $a_1,\ldots,a_k\in V$, $b_1,\ldots,b_k\in W$, and $r\in [0,1]$ such that $f(a_i,b_j)\geq r+\delta$ if $i\leq j$, and $f(a_i,b_j)\leq r$ if $i>j$. 
\end{definition}

Note that a binary relation $E\seq V\times W$ is $k$-stable (as defined in Remark \ref{rem:stable-graphs}) if and only if $\boldsymbol{1}_E$ is ${}^*\!(k,1)$-stable (if and only if $\boldsymbol{1}_E$ is ${}^*\!(k,\delta)$-stable for all $\delta>0$). As we previously observed, for ``discretely-valued" functions such as $\boldsymbol{1}_E$, ${}^*\!(k,\delta)$-stability implies $(\ell,\delta)$-stability for suitable $\ell$. On the other hand, there is a small discrepancy for arbitrary functions.

\begin{proposition}\label{prop:twostables}$~$
\begin{enumerate}[$(a)$]
\item Any $(k,\delta)$-stable $[0,1]$-valued function is ${}^*\!(k,\delta)$-stable.
\item For any $k\geq 1$ and $\delta'>\delta>0$, there is some $\ell\geq 1$ such that any ${}^*\!(k,\delta)$-stable $[0,1]$-valued function is $(\ell,\delta')$-stable. 
\end{enumerate}
\end{proposition}
\begin{proof}
Part $(a)$. This is straightforward to check.

Part $(b)$. Given $m,n\geq 1$ let $R_m(n)$ be an integer such that any $m$-coloring of ${[R_m(n)]\choose 2}$ admits a monochromatic subset of size $n$. Fix $k\geq 1$ and $\delta'>\delta>0$. Set $\epsilon=\delta'-\delta$ and fix $m\geq\epsilon\inv$. Set $\ell=R_2(R_m(2k+1))$. Suppose $f\colon V\times W\to [0,1]$ is not $(\ell,\delta')$-stable. We will show that $f$ is not ${}^*\!(k,\delta)$-stable.

Fix $a_1,\ldots,a_\ell\in V$ and $b_1,\ldots,b_\ell\in W$ such that $|f(a_i,b_j)-f(a_j,b_i)|\geq\delta'$ for all $i<j$. Consider a $2$-coloring of ${[\ell]\choose 2}$ according to whether $f(a_i,b_j)\geq f(a_j,b_i)+\delta'$ or $f(a_i,b_j)\leq f(a_j,b_i)-\delta'$. Set $n=R_m(2k+1)$. By choice of $\ell$, and after relabeling and reversing the order (if necessary), we may assume there are $a_1,\ldots,a_n\in V$ and $b_1,\ldots,b_n\in W$ such that $f(a_i,b_j)\geq f(a_j,b_i)+\delta'$ for all $i<j$. 

For $1\leq t\leq m$, set $I_t=[\frac{t-1}{m},\frac{t}{m}]$. Consider an $m$-coloring of ${[n]\choose 2}$ by the minimal $t$ such that $f(a_j,b_i)\in I_t$ (for $i<j$). By choice of $n$, and after relabeling, we may assume we have $1\leq t\leq m$, $a_1,\ldots,a_{2k+1}\in V$, and $b_1,\ldots,b_{2k+1}\in W$ such that for all $i<j$, $f(a_i,b_j)\geq f(a_j,b_i)+\delta'$ and $f(a_j,b_i)\in I_t$. Set $r=\frac{t}{m}$. Then for all $i<j$ we have $f(a_j,b_i)\leq r$ and $f(a_i,b_j)\geq r-\frac{1}{m}+\delta'\geq r+\delta$. Altogether, the sequences $(a_{2i})_{i=1}^k$ and $(b_{2i+1})_{i=1}^k$ witness that $f$ is not ${}^*\!(k,\delta)$-stable.
\end{proof}

In the context of a complete theory however, the discrepancy can be removed.

\begin{corollary}
Let $T$ be a complete theory in a continuous language $\cL$, and fix a $[0,1]$-valued $\cL$-formula $\varphi(x,y)$. 
Fix $\delta>0$, and assume $\varphi(x,y)$ is ${}^*\!(k,\delta)$-stable in $T$ for some $k\geq 1$. Then $\varphi(x,y)$ is $(\ell,\delta)$-stable in $T$ for some $\ell\geq 1$.
\end{corollary}
\begin{proof}
First one checks that ${}^*\!(k,\delta)$-stability in $T$ is an open condition in the sense of Lemma \ref{lem:stableopen}. So there is some $\delta_0<\delta$ such that $\varphi(x,y)$ is ${}^*\!(k,\delta_0)$-stable in $T$. By Proposition \ref{prop:twostables}, $\varphi(x,y)$ is $(\ell,\delta)$-stable in $T$ for some $\ell\geq 1$.
\end{proof}

The previous corollary also follows from \cite[Lemma 7.2]{BYU}, whose proof involves similar Ramsey arguments.

\end{appendices}

\end{document}